\documentclass[a4paper,11pt,reqno]{amsart}
\usepackage[utf8]{inputenc}
\usepackage{amsmath, amsfonts, verbatim, amssymb,latexsym, amsthm, color, enumerate, dsfont
, csquotes}
\usepackage[T1]{fontenc}
\usepackage{hyperref}
\usepackage[lmargin=2.5 cm,rmargin=2.5 cm,tmargin=3.5cm,bmargin=2.5cm,paper=a4paper]{geometry}
\usepackage{graphicx}
\numberwithin{equation}{section}

\newtheorem{theorem}{Theorem}[section]
\newtheorem{proposition}[theorem]{Proposition}
\newtheorem{lemma}[theorem]{Lemma}
\newtheorem{assumption}[theorem]{Assumption}
\newtheorem{definition}[theorem]{Definition}

\newtheorem{remark}[theorem]{Remark}
\newtheorem{notation}[theorem]{Notation}
\DeclareMathOperator{\supp}{supp}

\newcommand{\R}{\mathbb{R}}
\newcommand{\N}{\mathbb{N}}
\newcommand{\Dom}{\mathsf {Dom}}

\newcommand{\rig}{{\omega}}

\newcommand{\righ}{{\rm r}}

\newcommand{\lef}{{\rm \ell }}

\title[Robin Laplacian]{Tunneling for the Robin Laplacian\\
in smooth planar domains}
\author[B. Helffer]{Bernard Helffer}
\address[B. Helffer]{Laboratoire de Math\'ematiques d'Orsay, Univ. Paris-Sud, CNRS, Universit\'e 
Paris-Saclay, 91405 Orsay, France and Laboratoire Jean Leray, Universit\'e de Nantes, France.}
\email{bernard.helffer@math.u-psud.fr}
\author[A. Kachmar]{Ayman Kachmar}
\address[A. Kachmar]{Lebanese University, Department of Mathematics, Hadath, Lebanon.}
\email{ayman.kashmar@gmail.com}
\author[N. Raymond]{Nicolas Raymond}
\address[N. Raymond]{IRMAR - UMR6625, Universit\'e Rennes 1, CNRS, Campus de Beaulieu, F-35042 Rennes cedex, France}
\email{nicolas.raymond@univ-rennes1.fr}

\begin{document}

\maketitle
\begin{abstract}
We study the low-lying eigenvalues of the semiclassical Robin Laplacian in a smooth planar  domain with bounded boundary which is symmetric with respect to an axis. In the case when the curvature of the boundary of the domain attains its maximum at exactly two points away from the axis of symmetry, we establish an explicit asymptotic formula for the splitting of the first two eigenvalues. This is a rigorous derivation of the semiclassical tunneling effect induced by the domain's geometry. Our approach is close to the Born-Oppenheimer one and  yields, as a byproduct,  a Weyl formula of independent interest.
\end{abstract}
\maketitle
\section{Introduction}

The spectral theory of the Robin Laplacian has attracted a lot of interest in the last years, especially in the strong coupling regime or, equivalently, in the semiclassical limit. Many authors have been interested in the asymptotic estimate of the bound states of this operator. The Robin Laplacian actually shares common features with the electro-magnetic Laplacian, the Dirichlet Laplacian on waveguides or $\delta$ type perturbations of the Laplacian. These operators are often used to describe the physical properties of nanostructures (see for instance the review \cite{Ex08}).

In all these situations, numerous articles have revealed the role of the curvature  in the creation of eigenvalues or in the localization of the eigenfunctions.  At some point, the case with Robin boundary conditions may also recall the boundary attraction that occurs for the magnetic Neumann Laplacian (and that is related to the surface superconductivity). At the scale of nanostructures the symmetries are known to induce a tunneling effect. This paper aims at quantifying this effect for bidimensional structures described by the Robin Laplacian on a smooth domain.

\subsection{Definition of the operator}
Let $\Omega\subset\R^{2}$ be an open domain with boundary $\Gamma=\partial\Omega$. We will work under various assumptions on the domain $\Omega$. First, we  consider the following two assumptions.
\begin{assumption}\label{hyp.main1}
$\Omega$ is smooth with a bounded, regular boundary.
\end{assumption}
As examples we can think of bounded domains (convex sets, annuli) or unbounded domains like the complementary of a bounded convex closed set.

\begin{assumption}\label{hyp.main2}
The curvature $\kappa$ on the  boundary $\Gamma$ attains its maximum $\kappa_{\max}$ at a finite number $N$  of points on $\Gamma$ and these maxima are non degenerate.
\end{assumption}
In the case when $N=2$ in Assumption~\ref{hyp.main2}, we will carry out a  refined analysis valid under the following stronger (geometric) assumption:
\begin{assumption}\label{hyp.main3}~
\begin{enumerate}[\rm i)]
\item $\Omega$ is symmetric with respect to the $y$-axis.
\item The curvature $\kappa$ on the  boundary $\Gamma$ attains its maximum at exactly two points $a_1$ and $a_2$ which are not on the symmetry axis and belong to the same connected component of the boundary. We write
\[a_1=(a_{1,1},a_{1,2})\in\Gamma\quad\mbox{ and }\quad a_2=(a_{2,1}, a_{2,2})\in\Gamma\,,\]
such that $a_{1,1}>0$ and $a_{2,1}<0\,$.
\item  The second derivative of the curvature (w.r.t. arc-length) at $a_1$ and $a_2$ is negative.
\end{enumerate}
\end{assumption}
A simple example of a domain satisfying all the assumptions  is the full ellipse
\[\left\{(x,y)~:~\frac{x^2}{a^2}+\frac{y^2}{b^2} <1\right\}\,,\mbox{ with } 0<b<a\,.\]
The two points in the boundary of maximal curvature are $(\pm a,0)$.
The second example  is  the complementary:
\[\left\{(x,y)~:~\frac{x^2}{a^2}+\frac{y^2}{b^2} >1\right\}\,,\mbox{ with } 0<a <b\,.\]
The two points in the boundary of maximal curvature are $(\pm a,0)$.

This paper is devoted to the semiclassical analysis of the operator
\begin{equation}\label{eq:Lh} \mathcal
L_h=-h^2\Delta\,,\end{equation} with domain
\begin{equation}\label{eq:Lh-dom}
\Dom(\mathcal L_h)=\{u\in H^2(\Omega)~:~
{\nu\cdot h^{\frac{1}{2}}\nabla u-u=0}
{\rm ~on~}\Gamma\}\,,
\end{equation}
where $\nu$ is the {\it outward} pointing normal and $h>0$ is the semiclassical parameter.

The associated quadratic form is given by
\[\forall u\in H^1(\Omega)\,,\quad \mathcal Q_{h}(u)=\int_\Omega|h\nabla u|^ 2 dx-h^{\frac{3}{2}}\int_{\Gamma}|u|^2\,ds(x)\,,\]
where $ds$ is the standard surface measure on the boundary.

Let $(\mu_n(h))$ be the sequence of the Rayleigh quotients  of the operator $\mathcal L_h$.  It is known (see \cite{HK-tams, H, PP-eh}) that the  bottom of the essential spectrum (if any) is non negative and this implies that, for all $n\in\mathbb{N}$, $\mu_{n}(h)$ belongs to the discrete spectrum as soon as $h$ is small enough and that it is precisely the $n$-th eigenvalue of $\mathcal L_h$ counting multiplicities.

The goal of this paper is to analyze the low-lying eigenvalues of the operator $\mathcal L_h$ in the semiclassical regime as $h\to 0$.
The semiclassical analysis of the operator $\mathcal L_h$  naturally arises  from the analysis of the Robin Laplacian with a large negative parameter $\alpha$,
\[\Big(-\Delta^{R_\alpha},\Dom(-\Delta^{R_\alpha})\Big)\quad{\rm where}\quad \Dom(-\Delta^{R_\alpha})=\{u\in H^2(\Omega)~:~\nu\cdot\nabla u+\alpha u=0{\rm ~on~}\Gamma\}\,,\]
which has received a lot of attention  (cf. \cite{EMP, LP, Pan, HK-tams, PanP}). The operator $-\Delta^{R_\alpha}$ arises in several contexts, the long-time dynamics in a
reaction-diffusion process \cite{LOS}, and the critical temperature for enhanced surface superconductivity \cite{GS}.

Putting $\alpha = - h^{-\frac 12}$, we observe that $\alpha\to-\infty$ as $h\to0$ and the relation between the operators $\mathcal L_h$ and $-\Delta^{R_\alpha}$ is displayed as follows
\[\sigma(-\Delta^{R_\alpha})=h^{-2}\sigma(\mathcal L_h)\,.\]

\subsection{Known results}
In this subsection, we recall the state of the art for this Robin problem, especially  the spectral reduction of the operator $\mathcal L_h$ to an effective Hamiltonian on the boundary $\Gamma$. We will review an old result for the double well problem and apply it on the effective Hamiltonian.

\subsubsection{About the semiclassical Robin Laplacian}
As a consequence of the results in \cite{HK-tams, H, PP-eh} we have the following theorem.
\begin{theorem}\label{thm:old1}
Under Assumptions \ref{hyp.main1}-\ref{hyp.main2}  and suppose that, among the maximal points of $\kappa$, there are exactly $M$ points $(a_{j})_{j\in\{1,\ldots, M\}}$ where $\kappa''$ is maximal, then there exist  a function
$h\mapsto\epsilon(h)\in(0,\infty)$ such that
\[\lim_{h\to0_+}\epsilon(h)=0\,,\]
and an interval
\begin{equation}\label{eq:sp-int}
I_h=\big]-h-\kappa_{\max}h^{3/2}+\gamma\,h^{7/4}-h^{7/4}\epsilon(h), -h-\kappa_{\max}h^{3/2}+\gamma\,h^{7/4}+h^{7/4}\epsilon(h)\big[\,,\quad \gamma=\sqrt{\frac{-\kappa''(a_{1})}2}\,,
\end{equation}
such that, for $h$ small enough,
\[\sigma(\mathcal L_h)\cap I_h=\{\mu_1(h),\mu_2(h),\ldots, \mu_M (h)\}\,,\]
and
\[\mu_{M+1}(h)=-h-\kappa_{\max}h^{3/2}+\hat\gamma\,h^{7/4}+o(h^{7/4})\,,\qquad \hat\gamma=\min \left(3\gamma, \min_{j=M+1,\ldots, N}\sqrt{\frac{-\kappa''(a_{j})}2} \right)\,.\]
\end{theorem}
Weaker versions of this result were obtained in \cite{Pan} (and references therein, see also \cite{LP}). This result is related to \cite[Theorem 1.1]{FH} in the magnetic case.

\subsubsection{About semiclassical tunneling on the circle}
The aim of this article is to analyze the splitting $\mu_2(h)-\mu_1(h)$ under the symmetry Assumption \ref{hyp.main3} ($M=2$). We will see that the proof is easily reduced to the case when $\Gamma$ has only one component (the one, by assumption unique,  where $\kappa$ attains its maximum). As already observed in \cite{H} the candidate for the splitting is obtained by considering the splitting for the operator
\begin{equation}\label{hameff}
\mathcal M^{\mathsf{eff}}_h= -h-\kappa_{\max}h^{\frac{3}{2}}+h^2D^2_{s} + h^\frac 32  \mathfrak{v} (s)\,,\qquad \mathfrak{v}=\kappa_{\max}-\kappa\,,
\end{equation}
acting on the periodic functions in $L^2\left(\mathbb R /(2L) \mathbb Z\right)$, where
\[L=\frac{|\Gamma|}{2}\,,\]
and $s$ the arc-length.
Equivalently the operator $\mathcal M^{\mathsf{eff}}_h$  can be considered as the Schr\"odinger operator on the compact one dimensional manifold $\Gamma$. This is a double well problem which can be treated as a particular case of Helffer-Sj\"ostrand \cite{HSj} with the effective semiclassical parameter being $\hbar:=h^\frac 14$.
\begin{definition}\label{def.mueff}
We denote by $\mu^\mathsf{eff}_{j}(h)$ the $j$-th eigenvalue of $\mathcal M^{\mathsf{eff}}_h$ (counting multiplicities).
\end{definition}
Let us recall the splitting formula for the Schr\"odinger operator $\mathcal M^{\mathsf{circ}}_\hbar:=\hbar^2 D^2_{s} + \mathfrak v(s)$ on the circle of length $2L$ when $\mathfrak v$ has two symmetric non degenerate wells at say $s_{\mathsf{r}}$ and $s_{\mathsf{\ell }}$ with $\mathfrak v(s_{\mathsf{r}})=\mathfrak v(s_{\mathsf{\ell}})=0$ and $\mathfrak v''(s_{\mathsf{r}})=\mathfrak v''(s_{\mathsf{\ell}}) >0$.  We follow the exposition of \cite[\S 4.5]{H88} (see also \cite{Rob87}) but note that the formulas are established only for an example. In this paper, we will also use in many places the presentation of \cite{BHR-c}. Because there are two geodesics between the two wells the discussion will depend on the comparison between the lengths of these two geodesics.
For that purpose, let us introduce
\begin{equation}\label{defS}
\mathsf{S} =\min \left(\mathsf{S}_{\mathsf{u}},\mathsf{S}_{\mathsf{d}}\right)\,,\quad  \mathsf{S}_{\mathsf{u}}=\int_{[s_{\mathsf{r}},s_{\mathsf{\ell}}] } \sqrt{\mathfrak{v}(s)} ds\,,\quad  \mathsf{S}_{\mathsf{d}}=\int_{[s_{\mathsf{\ell}}, s_{\mathsf{r}}] } \sqrt{\mathfrak{v}(s)} ds\,,
\end{equation}
where $[p,q]$ denotes the arc joining $p$ and $q$ in $\Gamma$ counter-clockwise.

The splitting formula for the operator $\mathcal M^{\mathsf{circ}}_\hbar$ is obtained by adding the \enquote{upper} and \enquote{lower} contributions and reads
 \begin{equation}\label{formspl1}
 \lambda_2(\hbar) -\lambda_1(\hbar)
  = 4  \hbar^\frac 12 \pi^{-\frac 12} \gamma^{\frac12}  \left(\mathsf{A}_{\mathsf{u}} \sqrt{\mathfrak v(0)}e^{- \frac{\mathsf{S}_{\mathsf{u}}}{\hbar}}+\mathsf{A}_{\mathsf{d}} \sqrt{\mathfrak v(L)}e^{- \frac{\mathsf{S}_{\mathsf{d}}}{\hbar}}\right)+\mathcal{O}(\hbar^{\frac{3}{2}}e^{-\frac{\mathsf{S}}{\hbar}})\,,
  \end{equation}
  where
  \begin{align*}
  &\mathsf{A}_{\mathsf{u}}=\exp\left(-\int_{[s_{\mathsf{r}}, 0]} \frac{ (\mathfrak v^\frac 12 )' (s)+\gamma}{ \sqrt{\mathfrak v(s)}} ds\right)\,,\\
  &\mathsf{A}_{\mathsf{d}}=\exp\left(-\int_{[s_{\mathsf{\ell}}, L]} \frac{ (\mathfrak v^\frac 12 )' (s) -\gamma}{ \sqrt{\mathfrak v(s)}} ds\right)\,,\\
&\gamma=\left(\mathfrak v''(s_{\mathsf{r}})/2\right)^\frac 12=\left(\mathfrak v''(s_{\mathsf{\ell}})/2\right)^\frac 12\,.
  \end{align*}
  Then, for the particular model $\mathcal M^{\mathsf{eff}}_h$, we easily notice that
  \begin{equation}
  \mu_2^{\mathsf{eff}}(h) -\mu_1^{\mathsf{eff}}(h) = h^\frac 32 ( \lambda_2(\hbar) -\lambda_1(\hbar))\,,
  \end{equation}
 so that, under Assumption \ref{hyp.main3}, we have
  \begin{multline}\label{spl1}
 \mu_2^{\mathsf{eff}}(h) - \mu_1^{\mathsf{eff}}(h)   = 4  h^\frac{13}{8} \pi^{-\frac 12} \gamma^\frac 12  \left(\mathsf{A}_{\mathsf{u}} \sqrt{\mathfrak v(0)}\exp-\frac{\mathsf{S}_{\mathsf{u}}}{h^{\frac{1}{4}}}+\mathsf{A}_{\mathsf{d}} \sqrt{\mathfrak v(L)}\exp-\frac{\mathsf{S}_{\mathsf{d}}}{h^{\frac{1}{4}}}\right)\\
+\mathcal{O}\left(h^{\frac{13}{8}+\frac{1}{4}}\exp-\frac{\mathsf{S}}{h^{\frac{1}{4}}}\right)\,.
 \end{multline}
Let us notice here that the complete proof of \eqref{formspl1} provides a full asymptotic expansion and that the same holds for \eqref{spl1}. Note that, if we assume that $\mathfrak{v}$ is invariant under the symmetry exchanging the upper and lower parts, we have $\mathfrak v(0)=\mathfrak v(L)$, $\mathsf{S}_{\mathsf{u}}=\mathsf{S}_{\mathsf{d}}$ and  $\mathsf{A}_\mathsf{u}=\mathsf{A}_\mathsf{d}$.

\subsection{Statement of the main result}
The main result of this paper is the following.
\begin{theorem}\label{theo.main}
Under Assumptions~\ref{hyp.main1}~and~\ref{hyp.main3}, we have
\begin{equation}
\mu_2(h)-\mu_1(h) \underset{h\to 0}{\sim} \mu_2^{\mathsf{eff}}(h)-\mu_1^{\mathsf{eff}}(h)\,,
\end{equation}
where $\mu^{\mathsf{eff}}_{j}(h)$ is defined in Definition \ref{def.mueff} and where $\mu_2^{\mathsf{eff}}(h)-\mu_1^{\mathsf{eff}}(h)$ satisfies the asymptotic formula in \eqref{spl1}.
\end{theorem}
The result in Theorem~\ref{theo.main} shows a tunneling effect induced by the geometry of the domain (comparing with \cite{HSj5}, the boundary acts as the well and the points of maximal curvature as the mini-wells). This kind of reduction is also expected to be available for the magnetic Laplacian with a Neumann condition in smooth domains (see \cite{FH, FH-b}). However,  magnetic fields induce a lot of additional difficulties especially in obtaining the optimal decay estimates of the eigenfunctions.

Recently, magnetic WKB expansions  are established  in \cite{BHR}.  Note that, in superconductivity, computing  the splitting of the eigenvalues is useful to analyze the bifurcation from the normal state (cf. \cite[Lemma~13.5.4]{FH-b}).

When the domain $\Omega$ has corners and symmetries (e.g. the interior of an isosceles triangle), the tunneling effect is analyzed by Helffer-Pankrashkin in \cite{HPan}. One difference between the setting of Theorem~\ref{theo.main} and that in \cite{HPan} appears in the spectral reduction to the reference problems. In \cite{HPan}, the reference problem is a two-dimensional problem in an infinite sector which has an explicit groundstate. In this paper, the limiting reference problem is  a direct sum of two one-dimensional operators. To prove Theorem~\ref{theo.main}, we need to compare the eigenfunctions of the operator $\mathcal L_h$ with WKB approximate eigenfunctions (cf. Propositions~\ref{prop.WKB} and \ref{prop:WKB=gs}).

In higher dimensional domains,  a spectral reduction, modulo $\mathcal{O}(h^2)$, to an effective Hamiltonian on the boundary is done in \cite{PP-eh} (see also Section \ref{subsec.BO} where this reduction is explained). However, the analysis of the splitting as in Theorem~\ref{theo.main} requires additional estimates since we want to control exponentially small error terms.

In superconductivity, imposing a Robin condition, which is called in this context the de Gennes condition, models a superconductor surrounded by another normal/superconducting material (cf. \cite[Theorem.~1.2]{Ka-cocv}). In this context, we are naturally led to the analysis of the Robin Laplacian with a magnetic field where various regimes  occur according to the comparison between the intensity of the magnetic field and the Robin parameter (cf. \cite{Ka1, Ka2, Ka3}).

\subsection{Organization of the paper and strategy of the proofs}
Although it is easy to predict the statement in Theorem~\ref{theo.main} once  the effective Hamiltonian at the boundary is exhibited,
the proof of the formula is much more technical. It will follow the  steps outlined below:
\begin{enumerate}[-]
\item In Section \ref{section4}, we recall the known results related to the one dimensional situation.
\item In Section \ref{sec.tubular}, we recall why the first eigenfunctions are localized, in the Agmon sense, near the boundary (the boundary is a well). As a consequence, we replace the initial problem by a problem in a thin tubular neighborhood of the boundary. Then the inhomogeneity of the new operator leads to a rescaling in the normal variable in the Born-Oppenheimer spirit and the introduction of the effective semiclassical parameter $\hbar=h^{\frac{1}{4}}$.
\item In Section \ref{sec.miniwell}, we analyze one mini-well problems (i.e. with one point of maximal curvature). Note that this terminology is the one of \cite{HSj5} where the problem was to analyze miniwells inside a degenerate well. We briefly recall the WKB constructions of \cite{HK-tams} in Subsection \ref{subsec.WKB}. Then, we establish optimal Agmon estimates in the tangential direction (see Subsection \ref{subsec.tAgmon}).
\item In Section \ref{sec.WKBapprox}, we use the tangential estimates to prove that the first eigenfunctions are approximated (in the appropriate weighted space) by the WKB constructions. To do this, we are essentially led to use the same arguments as in dimension one (with respect to the tangential variable). Such estimates are closely related to the considerations of \cite{Mar}.
\item In Section \ref{sec.doublewell}, we analyze the interaction between the mini-wells and establish Theorem \ref{theo.main}.
\item Finally, in Section \ref{subsec.BO},  independently of the tunneling problem, we derive a Weyl asymptotic formula for the counting function inspired by the considerations of \cite[Chapter 13]{Ray} and related to the effective Hamiltonian $\mathcal{M}^\mathsf{eff}_{h}$.
\end{enumerate}

\section{Robin Laplacians in one dimension}\label{section4}

Before starting the proof of our main results, it is convenient to introduce three reference $1D$-operators and to determine their spectra. These models  naturally arise in our strategy of dimensional reduction and already appeared in \cite{HK-tams}.

\subsection{On a half line}
As simplest model, we start with the operator, acting on $L^2(\R_+)$, defined by
\begin{equation}\label{defH00}
\mathcal H_{0}=-\partial^2_{\tau}
\end{equation}
with domain
\begin{equation}
\Dom(\mathcal H_{0})=\{u\in H^2(\R_+)~:\,u'(0)=- u(0)\}\,.
\end{equation}
Note that this operator is associated with the quadratic form
\[
V_0 \ni u\mapsto \int_0^{+\infty} |u'(\tau)|^2\,d\tau\,  - |u (0)|^2\,,
\]
with $V_0 = H^1(0,+\infty)\,$.

The spectrum of this operator is
$\{-1\}\cup[0,\infty)$.
The eigenspace of the eigenvalue $-1$ is generated by the $L^2$-normalized function
\begin{equation}\label{eq:u0}
u_0(\tau)=\sqrt{2}\,\exp\left(-\tau\right)\,.
\end{equation}
We will also consider this operator in a bounded interval $(0,T)$ with $T$ sufficiently large and Dirichlet condition at $\tau=T$.

\subsection{On an interval}
Let us consider $T\geq 1$ and the self-adjoint operator acting on $ L^2(0,T)$ and defined by
\begin{equation}\label{eq:H0}
\mathcal H^{\{T\}}_{0}=-\partial^2_{\tau}\,,
\end{equation}
with domain,
\begin{equation}\label{eq:DomH0}
\Dom(\mathcal H^{\{T\}}_{0})=\{u\in H^2(0,T)~:~u'(0)=-u(0)\quad{\rm and}\quad u(T)=0\}\,.
\end{equation}
The spectrum of the operator $\mathcal H^{\{T\}}_{0}$ is purely discrete
and consists of a strictly increasing sequence of eigenvalues
denoted by $\left(\lambda_n\left(\mathcal H^{\{T\}}_{0}\right)\right)_{n\geq 1}$. This operator is associated with the quadratic form
\[
V^{\{T\}}_0\ni u\mapsto \int_0^{T} |u'(\tau)|^2\,d\tau\,  - |u (0)|^2\,,
\]
with $V^{\{T\}}_0 =\{v\in H^1(0,T)\,|\, v(T)=0\}$.\\
The next lemma gives the  localization of  the two first eigenvalues
$\lambda_1\left(\mathcal H^{\{T\}}_{0}\right)$ and $\lambda_2\left(\mathcal H^{\{T\}}_{0}\right)$ for large values of $T$.

\begin{lemma}\label{lem:1DL}
As $T \to +\infty$, there holds
\begin{equation}\label{eq:lwh}
\lambda_1(\mathcal H^{\{T\}}_{0})= - 1 + 4 \big(1+o(1)\big) \exp\big( - 2 T\big)\quad{\rm
and}\quad \lambda_2(\mathcal H^{\{T\}}_{0})\geq 0\,.
\end{equation}
\end{lemma}
\subsection{In a weighted space}
Let $B\in\R$, $T >0$ such that $BT < \frac 13$. Consider the self-adjoint
operator, acting on $L^2\big((0,T);(1-B\tau)d\tau\big)$ and defined by
\begin{equation}\label{eq:H0b}
\mathcal H^{\{T\}}_{B}=-(1-B\tau)^{-1}\partial_{\tau}(1-B\tau)\partial_{\tau}=-\partial^2_{\tau}+B(1-B\tau)^{-1}\partial_{\tau}\,,
\end{equation}
with domain
\begin{equation}\label{eq:domH0b}
\Dom(\mathcal H^{\{T\}}_{B})=\{u\in H^2(0,T)~:~u'(0)=-u(0)\quad{\rm and}\quad u(T)=0\}\,.
\end{equation}
The operator $\mathcal H^{\{T\}}_{B}$ is the Friedrichs extension in $L^2\big((0,T);(1-B\tau)d\tau\big)$ associated
with the quadratic form defined for $u\in V^{\{T\}}_h =H^1((0,T))\cap\{u(T)=0\}$, by
\[q^{\{T\}}_{B}(u)=\int_0^{T}|u'(\tau)|^2(1-B\tau)\,d\tau-|u(0)|^2\,.\]
The operator $\mathcal H^{\{T\}}_{B}$ is with compact resolvent. The strictly increasing
sequence of the eigenvalues of $\mathcal H^{\{T\}}_{B}$ is denoted by
$(\lambda_n(\mathcal H^{\{T\}}_{B})_{n\in \mathbb N^*}$.
It is easy to compare the spectra of $\mathcal H^{\{T\}}_{B}$ and  $\mathcal H^{\{T\}}_{0}$ as $B$ goes to $0$.
\begin{lemma}\label{lem:H0b}
There exists $T_0 >0$ and $C$  such that for all $T \geq T_0$, for all $B\in \left(-1/(3T), 1/(3T)\right)$ and $n\in\mathbb N^*$, there holds,
\[\left|\lambda_n(\mathcal H^{\{T\}}_{B})-\lambda_n(\mathcal H^{\{T\}}_{0})\right|\leq C |B| T\Big(\,\big|\lambda_n(\mathcal H^{\{T\}}_{0})\big|+1\Big)\,.\]
\end{lemma}
Then we notice that, for all $T>0$, the family $\left(\mathcal{H}^{\{T\}}_{B}\right)_{B}$ is analytic for $B$ small enough. In particular, its first eigenvalue $\lambda_{1}\Big(\mathcal{H}^{\{T\}}_{B}\Big)$ and the corresponding positive normalized eigenfunction $u_{B}^{\{T\}}$ are analytic functions of $B$.

\begin{lemma}\label{lem.analyticB}
There exists $T_0 >0$  such that for all $T \geq T_0$,   the functions $(-1/(3T), 1/(3T)) \ni B\mapsto \lambda_{1}\left(\mathcal{H}^{\{T\}}_{B}\right)$ and  $(-1/(3T), 1/(3T)) \mapsto u_{B}^{\{T\}}$ are analytic. 
\end{lemma}

\begin{proof}
The family  $\left(\mathcal{H}^{\{T\}}_{B}\right)_{B\in(-B_{0}, B_{0})}$ does not fulfill the conditions for type $(B)$  analytic
operators in the sense of  Kato since the parameter $B$  appears in the definition of (the norm of) the ambient Hilbert space.  Nevertheless, it
becomes so after using the change of function $u=(1-B\tau)^{-\frac{1}{2}}\tilde u$, since the new Hilbert space becomes
$L^2((0, T), d\tau)$, the form domain is still independent of the parameter and the expression of the operator depends on $B$ analytically:
\begin{equation}\label{eq.tildeH}
\widetilde{\mathcal H}_{B}^{\{T\}}=-(1-B\tau)^{-\frac{1}{2}}\partial_{\tau}(1-B\tau)\partial_{\tau}(1-B\tau)^{-\frac{1}{2}}=-\partial^2_{\tau}-\frac{B^2}{4(1-B\tau)^2}\,,
\end{equation}
with the new Robin condition at $0$ given by $\tilde u'(0)=\left(-1-\frac{B}{2}\right)\tilde u (0)$ and $\tilde u(T)=0$. The price to pay  is that the domain of the operator depends on $B$ through the $B$-dependent boundary condition.  Note that the associated quadratic form is defined on $H^1(0,T)$ by
\begin{equation}\label{eq.qtilde}
\tilde q_{B}^{\{T\}}(\psi)=\int_{0}^T |\partial_{\tau}\psi|^2\,d\tau-\int_{0}^T\frac{B^2}{4(1-B\tau)^2}|\psi|^2\,d\tau-\left(1+\frac{B}{2}\right)|\psi(0)|^2\,.
\end{equation}
\end{proof}

The next proposition states a two-term asymptotic expansion of the eigenvalue $\lambda_1(\mathcal H^{\{T\}}_{B})$.
\begin{proposition}\label{lem:H0b;l} There exists $T_0 >0$ and $C>0$  such that for all $T \geq T_0$, for all $B\in \left(-1/(3T), 1/(3T)\right)$ there holds,
\[\Big|\lambda_1(\mathcal H^{\{T\}}_{B})-(-1-B)\Big|\leq CB^2\,.\]
\end{proposition}

One will also need a decay estimate of $u_{B}^{\{T\}}$ that is a classical consequence of Proposition \ref{lem:H0b;l}, of the fact that the Dirichlet problem on $(0,T)$ is positive, and of an Agmon estimate.

\begin{proposition}\label{prop.AgmonuBT}There exists $T_0 >0$, $\alpha >0$ and $C>0$  such that for all $T \geq T_0$, for all $B\in \left(-1/(3T), 1/(3T)\right)$  there holds,
\[\|e^{\alpha\tau}u_{B}^{\{T\}}\|_{L^2\big((0,T);(1-B\tau)d\tau\big)}\leq C\,.\]
\end{proposition}

\begin{remark}
We will apply the results of this section with $T=Dh^{-r}$, $r\in(0,\frac 12)$,  $B=h^{\frac{1}{2}}\kappa$ and $h\in (0,h_0)$ for $h_0$ small enough.
\end{remark}

\section{Reduction to  a tubular neighborhood of the boundary}\label{sec.tubular}
\subsection{Agmon estimates}
As proved in \cite{HK-tams},  the eigenfunctions of the initial operator $\mathcal L_h$ are localized near the boundary and this localization is quantified by the following theorem:
\begin{theorem}\label{thm:dec}
Let $\epsilon_{0}\in(0,1)$ and $\alpha\in(0,\sqrt{\epsilon_{0}})$. There exist constants $C>0$  and $h_0\in(0,1)$ such
that, for  $h\in(0,h_0)$,  if $u_h$ is a normalized eigenfunction of $\mathcal L_h$ with eigenvalue $\mu\leq-\epsilon_{0}h$, then,
\[\int_\Omega \left(|u_h(x)|^2+h|\nabla u_h(x)|^2\right)\exp\left(\frac{2\alpha\, {\rm dist}(x,\Gamma)}{h^{\frac 1 2}}\right)\,dx\leq C\,.\]
\end{theorem}
Hence, this theorem is a quantitative version of the statement that the boundary is a well (in analogy with the Schr\"odinger model in \cite{HSj}) as $h\to 0$.
\subsection{Spectral reduction}\label{sec.spec.red}
We can  explicitly derive a reduction near each component of the boundary. From now on we assume for simplification that the boundary is connected. Given $\delta \in (0,\delta_0)$ (with $\delta_0>0$ small enough), we introduce the $\delta$-neighborhood of the boundary
\begin{equation}\label{eq.delta.neighbor}
\mathcal V_{\delta}=\{x\in\Omega~:~{\rm dist}(x,\Gamma)<\delta\}\,,
\end{equation}
and the quadratic form, defined on the variational space
\[V_{\delta}=\{u\in H^1(\mathcal V_{\delta})~:~u(x)=0\,,\quad\mbox{ for all } x\in\Omega \mbox{ such that } {\rm dist}(x,\Gamma)=\delta\}\,,\]
by the formula
\[\forall u\in V_{\delta}\,,\qquad\mathcal{Q}_{h}^{\{\delta\}}(u)=\int_{\mathcal{V}_{\delta}}|h\nabla u|^2 dx-h^{\frac{3}{2}}\int_{\Gamma}|u|^2\,ds(x)\,.\]
\begin{remark}\label{rem.choice-d}
In the following we will be led to take $\delta= Dh^{\rho}$ with $\rho\in\left(0,\frac{1}{4}\right]$. We will choose
\begin{enumerate}[-]
\item either $\rho <\frac 14$ with $D=1$,
\item or $\rho=\frac 14$ and $D> \mathsf{S}$ where $\mathsf S$ is defined in \eqref{defS} in order that the error term in \eqref{complem} is smaller than the tunneling effect, we want to measure.
\end{enumerate}
\end{remark}
Let us denote by $\mu^{\{\delta\}}_n(h)$ the $n$-th eigenvalue of the corresponding operator $\mathcal{L}_{h}^{\{\delta\}}$. It is then standard (cf. \cite{HSj}) to deduce from the Agmon estimates in Theorem~\ref{thm:dec} the following proposition.
\begin{proposition}\label{prop:red-bnd}
Let $\epsilon_{0}\in(0,1)$ and $\alpha\in(0,\sqrt{\epsilon_{0}})$.There exist constants $C>0$, $h_0\in(0,1)$ such that, for all $h\in(0,h_0)$, $\delta\in(0,\delta_{0})$, $n\geq 1$ such that $\mu_{n}(h)\leq-\epsilon_{0}h$,
\begin{equation}\label{complem}
\mu_n(h)\leq \mu^{\{\delta\}}_n(h)\leq \mu_n(h)+C\exp\left(-\alpha  \delta h^{-\frac 12} \right)\,.
\end{equation}
\end{proposition}
\subsection{Boundary coordinates}\label{sec:app}
Thanks to Proposition \ref{prop:red-bnd}, we can now work with the operator $\mathcal{L}^{\{\delta\}}_{h}$, with the choice of $\delta$ made in Remark \ref{rem.choice-d}.
Since the functions of its domain are supported near $\Gamma$ we will use the canonical tubular coordinates $(s, t)$ where $s$ is the arc-length and $t$ the distance to the boundary. We recall some elementary properties of these coordinates. Let
\begin{equation}\label{eq:s}
(\R/2L\mathbb{Z})\ni s\mapsto M(s)\in\Gamma
\end{equation}
 be a parametrization of $\Gamma$ (and thus we will always work with $2L$-periodic functions sometimes restricted to the interval $(-L,L]$).  The unit tangent vector of $\Gamma$ at the point $M(s)$ of the boundary is given by
\[
T(s):= M^{\prime}(s).
\]
We define the  curvature $\kappa(s)$ by the following identity
\[
T^{\prime}(s)=\kappa(s)\, \nu(s),
\]
where $\nu(s)$ is the unit vector, normal to the
boundary, pointing outward at the point $M(s)$. We choose the
orientation of the parametrization $M$ to be counter-clockwise, so
\[
\det(T(s),\nu(s))=1, \qquad \forall s\in (\R/2L\mathbb{Z})\,.
\]
We introduce the change of coordinates
\begin{equation}
\Phi:(\R/2L\mathbb{Z})
\times(0,t_{0})\ni  (s,t)\mapsto x= M(s)-t\, \nu(s)\in \mathcal{V}_{\delta_{0}}.
\end{equation}
The determinant of the Jacobian of $\Phi$ is
given by
\begin{equation}\label{eq:Jac-a}
a(s,t)=1-t\kappa(s).
\end{equation}
In the case of symmetry, we choose as origin of the parametrization the point $p_0$ defined as follows
\[\{p_0=(x_0,y_0),\bar p_0=(\bar x_0,\bar y_0)\}=\Gamma\cap\{x=0\}\quad {\rm and~}y_0>\bar y_0\,,\]
i.e. we suppose that $s(p_0)=0$ and $s(\bar p_0)=L$. This is illustrated in Figure~1.
\begin{figure}\label{fig}
\begin{center}
\includegraphics[width=10cm]{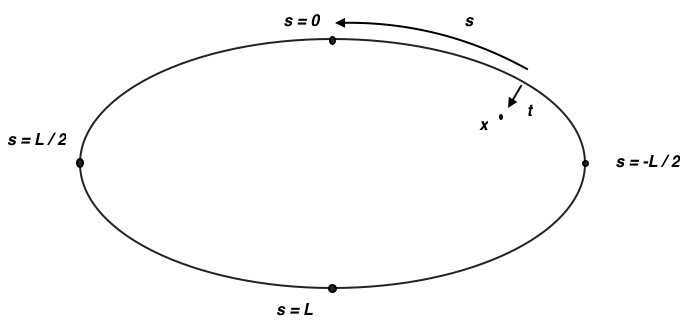} 
\end{center}
\caption{Illustration of the boundary coordinates $(s,t)$ for the point $x$. Note that for $t=0$ and $s\to-L$, the point $x$ approaches the boundary point $s=L$.}
\end{figure}

\subsection{The operator in a tubular neighborhood}
We can now express the operator in these new coordinates. To indicate that we work in the coordinates $(s, t)$, we put tildes on the functions. For all $u\in L^{2}(\mathcal V_{\delta_{0}})$, we define the pull-back function
\begin{equation}\label{eq:trans-st}
\widetilde u(s,t):= u(\Phi(s,t)).
\end{equation}
For all $u\in H^{1}(\mathcal{V}_{\delta_{0}})$, we have
\begin{equation}\label{eq:bc;n}
\int_{\mathcal{V}_{\delta_{0}}}|u|^{2}dx=\int|\widetilde u(s,t)|^{2}(1-t\kappa(s))\,ds dt\,,
\end{equation}
\begin{equation}\label{eq:bc;qf}
\int_{\mathcal{V}_{\delta_{0}}}|\nabla u|^{2}dx= \int \Big[(1-t\kappa(s))^{-2}|\partial_{s} \widetilde u|^{2} +|\partial_{t}\widetilde u|^{2}\Big](1-t\kappa(s))\,dsdt\,.
\end{equation}
The operator $\mathcal L^{\{\delta\}}_h$ is expressed in $(s,t)$ coordinates as
\[\mathcal L^{\{\delta\}}_h=-h^2a^{-1}\partial_s(a^{-1}\partial_s)-h^2a^{-1}\partial_t(a\partial_t)\,,\]
acting on $L^2(adsdt)$. In these coordinates, the Robin condition becomes
\[h^2\partial_tu=-h^{\frac 32}u\quad{\rm on}\quad t=0\,.\]
We introduce, for $\delta \in (0,\delta_0)$,
\begin{equation}\label{eq:red-bnd}
\begin{aligned}
&\widetilde{\mathcal V}_\delta=\{(s,t)~:~s\in ]-L, L]~{\rm and}~0<t< \delta\}\,, \\
&\widetilde V_\delta =\{u\in H^1(\widetilde{\mathcal V_\delta})~:~u(s,\delta)=0\}\,,\\
&\widetilde{\mathcal D}_\delta=\{u\in H^2(\widetilde{\mathcal V_{\delta}})\cap \widetilde V_\delta ~:~\partial_tu(s,0)=-h^{-\frac{1}{2}}u(s,0)\}\,,\\
&\widetilde{\mathcal{Q}}_{h}^{\{\delta\}}(u)=\int_{\widetilde{\mathcal V_\delta}}\Big(a^{-2}|h\partial_su|^2+|h\partial_tu|^2\Big)a\,dsdt-h^{\frac 32}\int |u(s,0)|^2\,ds\,,\\
&\widetilde{\mathcal L}_h^{\{\delta\}}= -h^2a^{-1}\partial_s(a^{-1}\partial_s)-h^2a^{-1}\partial_t(a\partial_t)\,.
\end{aligned}
\end{equation}
We now take \begin{equation}
\delta= D h^{\rho}\,,
\end{equation}
and write simply $\widetilde{\mathcal L}_h$ for $\widetilde{\mathcal L}_h^{\{\delta\}}$.
The operator $\widetilde{\mathcal L}_h$ with domain $\widetilde{\mathcal D}$ is the self-adjoint operator defined via the closed quadratic form $\widetilde{\mathcal V}_{\rho}\ni u\mapsto \widetilde{\mathcal{Q}}_{h}(u)$ by Friedrich's theorem.

\subsection{The rescaled operator}
In order to perform the analysis and to compare with existing strategies, it will be convenient to work with a rescaled version of $\widetilde{\mathcal{L}}_{h}$. We introduce the rescaling
\[(\sigma,\tau)=(s,h^{-\frac 12}t)\,,\]
the new semiclassical parameter $\hbar=h^{\frac 1 4}$ and the new weight
\begin{equation}\label{eq:Jac-a'}
\widehat a(\sigma,\tau)=1-h^{\frac{1}{2}}\tau\kappa(\sigma)\,.
\end{equation}
We consider rather the operator
\begin{equation}\label{eq.Lh-hat}
\widehat{\mathcal{L}}_{\hbar}=h^{-1}\widetilde{\mathcal{L}}_{h}\,,
\end{equation}
acting on $L^2(\widehat a\, d\sigma d\tau)$ and expressed in the coordinates $(\sigma,\tau)$. As in \eqref{eq:red-bnd}, we let
\begin{equation}\label{eq:dom-Lh-hat}
\begin{aligned}
&\widehat{\mathcal V}_{T }=\{(\sigma,\tau)~:~\sigma\in ]-L,L]~{\rm and}~0<\tau<
 T  \}\,, \\
&\widehat V_{T }=\{u\in H^1(\widehat{\mathcal V_{T }})~:~u(\sigma, T ) =0\}\,,\\
&\widehat{\mathcal D}_{T }=\{u\in H^2(\widehat{\mathcal V_{T }})\cap \widehat V_{T }~:~\partial_\tau u(\sigma,0)=-u(\sigma,0)\}\,,\\
&\widehat{\mathcal{Q}}_{\hbar}^{T }(u)=\int_{\widehat{\mathcal V_{T }}}\Big(\widehat a^{-2}\hbar^4|\partial_\sigma u|^2+|\partial_\tau u|^2\Big)\widehat a \, d\sigma d\tau-\int_{-L}^L |u(\sigma,0)|^2\,d\sigma\,,\\
&\widehat{\mathcal{L}}_\hbar^{T }=-\hbar^4\,\widehat a^{-1}\partial_\sigma\,\widehat a^{-1}\partial_\sigma-\widehat\, a^{-1}\partial_\tau\widehat a\partial_\tau\,.
\end{aligned}
\end{equation}
\begin{remark}
We then specify the analysis for $$T  = h^{-\frac 12} \delta=Dh^{\rho-\frac{1}{2}}  \,$$ and omit the reference to $T $.
\end{remark}

\section{Simple mini-well}\label{sec.miniwell}
This section is devoted to the analysis of the eigenfunctions when the curvature has a unique non degenerate maximum (i.e. Assumptions \ref{hyp.main1} and \ref{hyp.main2} with $M=1$). We will investigate both the WKB constructions and the accurate approximation of the eigenfunctions in such a situation. For that purpose, we will constantly work with the operator $\widehat{\mathcal L}_{\rig}$ defined in the sequel.
\subsection{Definition of the simple mini-well operator}\label{subsec.defminiwell}
Let $\omega$ be an (open) interval in the circle of length $2L$ identified with the interval $(-L,L]$. We can view $\omega $ as a (curved) segment in the boundary of $\Omega$ by means of the parametrization in \eqref{eq:s}. The operator $\widehat{\mathcal L}_{\rig, \hbar}^T=\widehat{\mathcal L}_{\rig}$ is defined as follows. We assume that $\omega$  contains a unique point $s_{\omega}$ of maximum curvature (i.e. $\kappa(s_{\omega})=\kappa_{\max})$  that is non degenerate. The form domain $\widehat{V}_{\rig}$ and the domain $\widehat{\mathcal D}_{\rig}$ of this operator are defined as follows,
\begin{equation}\label{eq:dom-right}
\begin{aligned}
&\widehat{\mathcal  V}_{\rig}=\omega\times(0,T)\,,\\
&\widehat {V}_{\rig}=\{u\in H^1(\widehat{\mathcal V}_{\rig})~:~u=0~{\rm on~}\tau=T   {\rm ~and~}\partial\omega\times(0,T)\}\,,\\
&\widehat{\mathcal{D}}_{\rig}=\{u\in H^2(\widehat{\mathcal V}_{\rig})\cap \widehat{V}_{\rig}~:~\partial_\tau u=-u~{\rm on~}\tau=0\}\,.
\end{aligned}
\end{equation}
The operator $\widehat{\mathcal L}_{\rig}$ is the self-adjoint operator on $L^2(\widehat{\mathcal V}_{\rig}; \widehat a\, d\sigma d \tau)$ with domain $\widehat{\mathcal{D}}_{\rig}$ and
\begin{equation}\label{eq:op-right}
\widehat{\mathcal L}_{\rig}=-\hbar^4\,\widehat a^{-1} (\partial_\sigma\widehat a^{-1})\partial_\sigma-\widehat a^{-1}(\partial_\tau\widehat a)\partial_\tau\,.
\end{equation}
We denote by $\mu_{\rig}(\hbar)$ its lowest eigenvalue. 
\begin{definition}
The corresponding positive and $L^2$-normalized eigenfunction is denoted by $\phi_{\hbar,\rig}$.
\end{definition}
Let $\mu_{2,\rig}(\hbar)$ be the second eigenvalue of the operator $\widehat{\mathcal L}_{\rig}$. The analysis in \cite{HK-tams} yields that, for $\hbar$ small, $\mu_\rig(\hbar)$  is a simple eigenvalue  and
\begin{equation}\label{eq:sp-gap.1well}
\mu_{2,\rig}(\hbar)-\mu_{\rig}(\hbar)=3\gamma\,\hbar^{7/4}+\hbar^{7/4}o(1)\quad\mbox{ as }\hbar\to0_+\,,
\end{equation}
where $\gamma=\sqrt{\frac{-\kappa''(s_{\omega})}2}$.

\subsection{Reminder of the WKB constructions}\label{subsec.WKB}

\subsubsection{Statements}
In this section, we recall the WKB construction of \cite{HK-tams} in the spirit of the paper by Bonnaillie-No\"el--H\'erau--Raymond \cite{BHR} (see also the classical references about the Born-Oppenheimer approximation \cite{BOp,Lef, Mar}).

\begin{proposition}\label{prop.WKB}
There exists a sequence of smooth functions $(a_{j})$ such that the following holds. We consider the formal series (or a smooth realization constructed by a Borel procedure)
\begin{equation}\label{eq.WKB-quasimode}
\Psi_{\hbar,\rig}(\sigma,\tau)\sim \hbar^{-\frac{1}{4}}e^{-\Phi_{\rig}(\sigma)/\hbar}\sum_{j\geq 0} \hbar^j a_{j}(\sigma,\tau)\,,
\end{equation}
where
\begin{enumerate}[\rm i)]
\item $\Phi_{\rig}$ is the Agmon distance to the well at $\sigma=s_{\rig}$ of the effective potential
\[\mathfrak v (\sigma)=\kappa_{\max}-\kappa(\sigma)\]
and defined by the formula
\[
 \Phi_{\rig}(\sigma)=\int_{[s_{\rig}, \sigma]}\sqrt{\mathfrak v (\tilde\sigma)}\,  d\tilde\sigma\,,\]
\item $a_{0}$ is in the form $a_{0}(\sigma,\tau)=\xi_{0,\rig}(\sigma)u_{0}(\tau)$ where
\[u_{0}(\tau)=\sqrt{2} e^{-\tau}\,,\]
and
\[\xi_{0,\rig}(\sigma)=\xi_0(\sigma)=\left(\frac{\gamma}{\pi}\right)^{\frac{1}{4}}\exp\left(-\int_{s_{\rig}}^\sigma \frac{\Phi_{ \omega}{''}-\gamma}{2\Phi_{\omega}{'}}d\tilde\sigma\right)\]
is the solution of the transport equation of the effective Hamiltonian
\[\Phi_{\rig}'\partial_{\sigma} \xi_{0}+\partial_{\sigma}(\Phi'_{\rig} \xi_{0})=\gamma\xi_{0}\,, \qquad\mbox{ with }\gamma=\sqrt{\frac{-\kappa''(s_{\omega})}{2}}\,.\]
\item For $j\geq 1$, $a_j(\sigma,\tau)$ is a linear combination of functions of the form
\[f_{j,k}(\sigma)g_{j,k}(\tau)\,,\quad f_{j,k}\in \mathcal{C}^\infty\big(\omega \big)~{\rm and~}g_{j,k}\in \mathcal S(\overline{\R_+})\,.\]
\item The formal series $\Psi_{\hbar,\rig}$ satisfies
\[e^{\Phi_{\rig}/\hbar}\left(\widehat{\mathcal L}_{\rig}-\mu\right)\Psi_{\hbar,\rig}=\mathcal{O}(\hbar^{\infty})\,,\]
where $\mu$ is an asymptotic series in the form
\begin{equation}\label{eq:Tay-mu}
\mu\sim-1-\kappa_{\max}\hbar^2+\gamma\hbar^3+\sum_{j\geq 4} \mu_{j}\hbar^{j}\,.
\end{equation}
This series is the Taylor series of the first eigenvalue  $\mu_{\rig}(\hbar)$.
\end{enumerate}
\end{proposition}
In the previous proposition, we have used the following notation.
\begin{notation} \label{formalfunctions}
We write $a(\sigma,\tau ;\hbar)\underset{\hbar\to 0}{\sim}\sum_{j\geq 0}a_{j}(\sigma,\tau) \hbar^j$ when for all $J\geq 0$, $\alpha \in \N^{2}$ and all compact $K\subset\omega\times\overline{\R_{+}}$, there exist $\hbar_{J,\alpha, K}>0$ and $C_{J,\alpha, K}>0$ such that for all $\hbar\in(0,\hbar_{J,\alpha, K})$, we have, on $K$,
\[\Big|\partial^{\alpha} \Big( a(\sigma,\tau ; \hbar)-\sum^J_{j= 0}a_{j}(\sigma,\tau)\hbar^j\Big)\Big|\leq C_{J,\alpha, K}\hbar^{J+1}\,.\]
We also write $a = \mathcal{O}(\hbar^\infty)$ when all the coefficients in the series are zero.
\end{notation}

\begin{remark}\label{cor.WKB}
In the sequel, it will be convenient to work with a truncated version of $\Psi_{\hbar, \omega}$. Let $\check\omega$ be an open interval such that $s_{\omega}\in \check\omega\subset \overline{\check \omega}\subset\omega$ and
\begin{equation}\label{eq.WKB-quasimode'}
\psi_{\hbar,\omega,{\check\omega}}(\sigma,\tau)=\chi_{{\check\omega}}(\sigma)\chi(T^{-1}\tau)\Psi_{\hbar,\rig}(\sigma,\tau)\,,
\end{equation}
where
\begin{enumerate}[\rm i)]
\item $\chi$ is a smooth function cut-off function with compact support being $1$ near $0$\,;
\item $\chi_{{\check\omega}}\in C_c^\infty(\rig)$ is a smooth cut of function satisfying  $0\leq \chi_{{\check\omega}}\leq 1$ and $\chi_{{\check\omega}}=1$ on  ${\check \omega}$.
\end{enumerate}
The truncated function  $\psi_{\hbar,{\check\omega}}$ satisfies
\[e^{\Phi_{\rig}/\hbar}\left(\widehat{\mathcal L}_{\rig}-\mu\right)\psi_{\hbar,\omega, {\check \omega}}=\mathcal{O}(\hbar^\infty)\,,\qquad \mbox{ \rm in }L^2(\widehat{\mathcal  V}_{{\check \omega}})\,,\]
\[\forall~j\in\{1,2\}\,,\quad
e^{\Phi_{\rig}/\hbar}\partial_\tau^j\left(\widehat{\mathcal L}_{\rig}-\mu\right)\psi_{\hbar,\omega, {\check \omega}}= \mathcal{O}(\hbar^\infty)\,,\qquad \mbox{ \rm in }L^2(\widehat{\mathcal  V}_{{\check \omega}})\,.\]
In the sequel we will use that $\check \omega$ and $\omega$ can be chosen as large as we want, as soon as $\omega$ only contains one mini-well  and $\check \omega$ satisfies the above condition.
\end{remark}

\subsubsection{Proof}
Let us just explain the main steps in the proof of Proposition \ref{prop.WKB}.
Thanks to a formal Taylor expansion, we find the following expansion of the operator $\widehat{\mathcal{L}}_\hbar$,
\begin{multline*}
\widehat{\mathcal{L}}_\hbar\sim -\partial_\tau^2-\hbar^4\partial_\sigma^2+2\hbar^{6}\tau\kappa(\sigma)\partial_\sigma^2+\hbar^{2}\kappa(\sigma)\partial_\tau+\hbar^{6}\tau\kappa'(\sigma)\partial_\sigma\\
-\sum_{j=1}^\infty c_j \hbar^{2j+4}\tau^j(\kappa(\sigma))^j\partial_\sigma^2+\sum_{j=1}^\infty \hbar^{2j+2}\tau^j(\kappa(\sigma))^{j+1}\partial_\tau-\kappa'(\sigma)\sum_{j=1}^\infty \hbar^{2j+6}d_j\tau^j(\kappa(\sigma))^j\partial_\sigma\,.
\end{multline*}
 We introduce the (formal) conjugate operator
\[\widehat{\mathcal{L}}_\hbar^\vartheta:=\exp \left( \frac{\vartheta (\sigma)}{\hbar}\right)  \,\widehat L_h\, \exp \left(- \frac{\vartheta (\sigma)}{\hbar}\right) \,,\]
and write
\begin{equation}\label{eq:WKB'}
\left( \widehat{\mathcal{L}}_\hbar^\vartheta - \mu\right) \left( \sum_\ell a_\ell (\sigma,\tau) \hbar^\ell\right)\sim 0 \,.
\end{equation}
We (formally) expand the operator $\widehat{\mathcal{L}}_\hbar^\vartheta$ as follows
\begin{equation}\label{eq:wkbL}
\widehat{\mathcal{L}}_\hbar^\vartheta \sim  \sum_{\ell=0}^\infty Q_\ell^\vartheta \,\hbar^\ell \,,
\end{equation}
with in particular
\begin{align*}
 &Q_0^\vartheta= -\partial_\tau^2\,,\\
 &Q_1^\vartheta =0\,,\\
 &Q_2^\vartheta  =  \kappa(\sigma) \partial_\tau  - \vartheta'(\sigma)^2 \,,\\
 &Q_3^\vartheta = 2 \vartheta'(\sigma)  \partial_\sigma + \vartheta'' (\sigma)\,,\\
 &Q_4^\vartheta = -\partial_\sigma^2+c_3\tau^3\kappa(\sigma)^3+\tau^3(\kappa(\sigma))^4\partial_\tau\,.\\
\end{align*}
We then rearrange all  the terms in \eqref{eq:WKB'} in the form of power series in $\hbar$ and select $\vartheta$,  $a_\ell (\sigma,\tau)$ and $\mu_\ell$ by  expressing the cancellation of each term of the formal series. The vanishing of the  coefficient of $\hbar^0$ yields the equation,
 \[(Q_0^\vartheta-\mu_0) a_0(\sigma,\tau) =0\,.\]
We have $Q_0^\vartheta=\mathsf{Id}\otimes \mathcal H_{0}$ on $L^2(\mathbb R_\sigma\times \mathbb R_{+,\tau})$. This leads us naturally (considering the operator $\mathcal H_{0}$ introduced in \eqref{defH00}) to the choice
\[\mu_0=-1\quad{\rm and}\quad a_0 (\sigma,\tau) = \xi_0(\sigma) u_0(\tau)\,.\]
Since $Q_1^\vartheta=0$, the vanishing of the  coefficient of $\hbar^1$ in \eqref{eq:wkbL}   yields
\[(Q_0^\vartheta -\mu_0) a_1(\sigma,\tau)-\mu_1a_1(\sigma,\tau) =0\,.\]
This leads us to the natural choice $\mu_1=0$ and
\[a_1(\sigma,\tau)=\xi_1(\sigma) u_0(\tau)\,.\]
We look at the coefficient of $\hbar^{2}$ and obtain
\[(Q_0^\vartheta-\mu_0) a_2 + (Q_2^\vartheta -\mu_2) a_0 =0\,.\]
Remembering that $\mu_0=-1$ and $a_0(\sigma,\tau)=\xi_0(\sigma)u_0(\tau)$, we get
\[(Q_0^\vartheta+1) a_2 =-u_0(\tau) (\kappa(\sigma)- \vartheta'(\sigma)^2-\mu_2) \xi_0(\sigma)\,.\]
By using the Fredholm condition with respect to $\tau$, we get the eikonal equation
\begin{equation}\label{eq:eik}
- \kappa(\sigma) - \vartheta'(\sigma)^2 -\mu_2 =0\,.
\end{equation}
Consequently, we  take $\mu_2=-\kappa(s_{\rig})$ and get $\vartheta'(s_{\rig})=0$ and we consider the solution such that $\vartheta''(s_{\rig})>0$. This gives
\begin{equation}\label{eq:v''}
\vartheta''(s_{\rig})=\sqrt{\frac{-\kappa''(s_{\omega})}{2}}\,,
\end{equation}
and
\begin{equation}\label{eq:v}
\vartheta(\sigma)=\int_{[s_{\omega},\sigma]}\sqrt{\kappa_{\max}-\kappa(\tilde\sigma)}\,d\tilde\sigma=\Phi_\rig(\sigma)\,,
\end{equation}
where $[s_{\omega},\sigma]$ is the segment joining $s_{\omega}$ and $\sigma$ counter-clockwise (the integral may also be understood as the Lebesgue integral on a measurable set, independently from the representation of the set).

We deduce that $a_2$ is in the form
\[a_2(\sigma,\tau)=\xi_2(\sigma) u_0(\tau)\,.\]
Now we look at the coefficient of $\hbar^{3}$ in
\eqref{eq:wkbL}. This yields
\[(Q_0^\vartheta+1) a_3 + (Q_2^\vartheta-\mu_2) a_1 + (Q_3^\vartheta -\mu_3)a_0 =0\,.\]
Using \eqref{eq:eik}, we see that the term $(Q_2^\vartheta-\mu_2) a_1$ vanishes and thus we get
\[(Q_0^\vartheta+1) a_3=- (Q_3^\vartheta -\mu_3)a_0\,.\]
For each fixed $\sigma$, the Fredholm condition implies that
\[\langle (Q_3^\vartheta -\mu_3)a_0, u_{0}\rangle_{L^2(\R_{+,\tau})}=0\,,\]
that is
\begin{equation}\label{eq:tr}
 2 \vartheta'(\sigma) \xi_0'(\sigma) + (\vartheta'' (\sigma) - \mu_3) \xi_0(\sigma)=0\,.
\end{equation}
Since we look for smooth solutions at $s_{\rig}$ and for the smallest possible $\mu_{3}$, the linearization at $\sigma=s_{\rig}$ leads to
\[\mu_3 = \sqrt{\frac{-\kappa''(s_{\omega})}{2}}=\gamma\,.\]
We can determine $\xi_0$  by solving \eqref{eq:tr} in a neighborhood of $\sigma=s_{\rig}$ and find
\[
\xi_0(\sigma) =\left(\frac{\gamma}{\pi}\right)^{\frac{1}{4}}\exp\left(-\int_{s_{\rig}}^\sigma \frac{\vartheta''-\gamma}{2\vartheta'}d\tilde\sigma\right)\,.
\]
where the constant is chosen to get a $L^2$-normalized quasimode (modulo $\hbar$).
 Then, we are led to choose
 \[a_3(\sigma,\tau) = \xi_3(\sigma) u_0(\tau)\,.\]
This construction may be continued at any order.

\subsection{Tangential Agmon's estimates}\label{subsec.tAgmon}
In this subsection, we derive Agmon's estimates for the eigenfunctions of the operator $\widehat{\mathcal L}_{\rig}$ with domain $\widehat{\mathcal D}_{\rig}$ and form domain $\widehat{V}_{\rig}$ introduced in \eqref{eq:dom-right}. Let us start with the following elementary lemma that is related to the Born-Oppenheimer approximation.
\begin{lemma}\label{lem:lb-qf}
There exist constants $C>0$ and $\hbar_0\in(0,1)$ such that, for all
$\hbar\in(0,\hbar_0)$ and $u\in\widehat{V}_{\rig}$,
\[
\widehat{\mathcal Q}_{\rig}(u)\geq \int_{\widehat {\mathcal V}_\rig}\widehat a^{-2}\hbar^4|\partial_\sigma u|^2\,\widehat{a}\, d\sigma d\tau +\int_{\widehat {\mathcal V}_\rig} \left(-1-\kappa_{\max}\hbar^{2}+\hbar^2 \mathfrak v (\sigma)-C\hbar^{4}\right)|u|^2\,\widehat{a}\, d\sigma d\tau\,.
\]
\end{lemma}
\begin{proof}
Using \eqref{eq:dom-Lh-hat}, we have
\begin{equation}\label{eq.Qrig}
\widehat{\mathcal Q}_{\rig}(u)=\int_{-L}^L \int_{0}^{T  }\Big(\widehat a^{-2}\hbar^4|\partial_\sigma u|^2+|\partial_\tau u|^2\Big)\widehat a \, d\tau d\sigma-\int_{-L}^L |u(\sigma,0)|^2\,d\sigma \\
\end{equation}
Recall the operator in $\mathcal H^{\{T\}}_{B}$ in \eqref{eq:H0b}. By
a simple scaling argument  and the min-max principle, we have
\begin{equation}\label{eq:imp;est}
\int_0^{T  }|\partial_\tau u|^2\,\widehat{a}d\tau-|u(\sigma,0)|^2 \geq \lambda_1(\mathcal
H^{\{T\}}_{B})\int_0^{T}|u|^2\,\widehat{a}d\tau\,,\qquad T=h^{\rho- \frac{1}{2}}\,,\quad B=h^{\frac{1}{2}}\kappa(\sigma)\,.
\end{equation}
Thanks to  Proposition~\ref{lem:H0b;l}, we deduce the lower bound since $\mathfrak v =\kappa_{\max}-\kappa$.
\end{proof}
From Lemma \ref{lem:lb-qf}, we may deduce some accurate tangential Agmon  estimates satisfied by $\phi_{\hbar,\rig}$. We will often use the following notation.
\begin{notation}\label{not.Bhat}
For $\varrho\in(0, L)$, we let
\[\mathcal B_{\rig}(\varrho)=(-\varrho+s_\rig,\varrho +s_\rig)\quad{\rm and}\quad \widehat{\mathcal B}_\rig(\varrho)=\mathcal B_{\rig}(\varrho)\times (0,T)\,.\]
\end{notation}

Let us first state a proposition that will be convenient in the sequel.
\begin{proposition}\label{prop.Agmon}
Suppose that $T= Dh^{-\frac14}$ and $D>\mathsf S$. Let $\Phi$ be a Lipschitzian function that is a subsolution of the eikonal equation:
\[\mathfrak v(\sigma)-|\Phi'(\sigma)|^2\geq 0\,,\qquad \forall \sigma\in\omega\,,\]
and let us assume that there exist a non decreasing function $\R_{+}\ni R\mapsto M(R)\in\R_{+}$ tending to $+\infty$ as $R\to+\infty$, a positive constant $\hbar_{0}$ such that, for all $\hbar\in (0,\hbar_{0})$, and $R>0$,
\begin{align*}
\mathfrak v (\sigma)-|\Phi'(\sigma)|^2&\geq M(R)\hbar\,,                         &\forall \sigma\in\omega\cap\complement\mathcal{B}_{\rig}(R\hbar^{\frac{1}{2}})\,,\\
|\Phi(\sigma)|                        &\leq M(R)\hbar\,,                          &\forall\sigma\in\mathcal{B}_{\rig}(R\hbar^{\frac{1}{2}})\,.
\end{align*}
Then, there exist $R_{0}, C>0$ and $\tilde\hbar_{0}\in(0,\hbar_{0})$ such that the following holds. For all $R\geq R_{0}$, $C_{0}\in(0,\frac{M(R)}{2})$, $\hbar\in(0,\tilde\hbar_{0})$, $z\in[-1-\kappa_{\max}\hbar^2, -1-\kappa_{\max}\hbar^2+C_{0}\hbar^3]$, $u\in\widehat{\mathcal{D}}_{\rig}$,
\begin{equation}\label{eq.estimate-weight1}
\hbar^3\|e^{\Phi/\hbar}u\|_{L^2(\widehat{\mathcal  V}_{\rig})}\leq C\|e^{\Phi/\hbar}\left(\widehat{\mathcal{L}}_{\rig}-z\right)u\|_{L^2(\widehat{\mathcal  V}_{\rig})}+C\hbar^3\|u\|_{L^2(\widehat{\mathcal  V}_{\rig}\cap\hat{\mathcal{B}}_{\rig}(R\hbar^{\frac{1}{2}}))}\,,
\end{equation}
and
\begin{equation}\label{eq.estimate-weight2}
\hbar^4\|\partial_{\sigma}(e^{\Phi/\hbar}u)\|^2_{L^2(\widehat{\mathcal  V}_{\rig})}\leq C\hbar^{-3}\|e^{\Phi/\hbar}\left(\widehat{\mathcal{L}}_{\rig}-z\right)u\|^2_{L^2(\widehat{\mathcal  V}_{\rig})}+C\hbar^3\|u\|^2_{L^2(\widehat{\mathcal  V}_{\rig}\cap\hat{\mathcal{B}}_{\rig}(R\hbar^{\frac{1}{2}}))}\,.
\end{equation}
\end{proposition}
\begin{proof}
By the usual Agmon formula, we get
\[\langle\widehat{\mathcal{L}}_{\rig}u,e^{2\Phi/\hbar}u\rangle=\widehat{\mathcal{Q}}_{\rig}(e^{\Phi/\hbar}u)-\hbar^2\int_{\widehat{\mathcal  V}_{\rig}} \widehat{a}^{-2}|\Phi'|^2e^{2\Phi/\hbar}|u|^2\, \widehat{a}\,d\sigma d\tau\,.\]
By Lemma \ref{lem:lb-qf}, we deduce
\begin{multline*}
\langle\widehat{\mathcal{L}}_{\rig}u,e^{2\Phi/\hbar}u\rangle\geq \int_{\widehat {\mathcal V}_\rig}\widehat a^{-2}\hbar^4|\partial_\sigma (e^{\Phi/\hbar}u)|^2\,\widehat{a}\, d\sigma d\tau\\
+\int_{\widehat {\mathcal V}_\rig} \left(-1-\kappa_{\max}\hbar^{2}+\hbar^2(\mathfrak v -\widehat{a}^{-2}|\Phi'|^2)-C\hbar^{4}\right)|e^{\Phi/\hbar}u|^2\,\widehat{a}\, d\sigma d\tau\,.
\end{multline*}
Note that, for all $(\sigma,\tau)\in\widehat{\mathcal  V}_{\rig}$, $|\hbar^2\kappa(\sigma)\tau|\leq D\hbar^{4\rho}|\kappa|_{\infty}$. Thus there exists $\tilde D,\tilde\hbar_{0}>0$ (depending only on $\rho$, $D$ and $|\kappa|_{\infty}$) such that, for all $\hbar\in(0,\tilde\hbar_{0})$ and $(\sigma,\tau)\in\widehat{\mathcal  V}_{\rig}$, 
\[\left|\widehat{a}^{-2}-1\right|\leq \tilde D\hbar^{4\rho}\,.\] 
This leads to choose $$\rho=\frac{1}{4}$$ and to the lower bound
\begin{multline*}
\langle(\widehat{\mathcal{L}}_{\rig}-z)u,e^{2\Phi/\hbar}u\rangle\geq  \int_{\widehat {\mathcal V}_\rig}\widehat a^{-2}\hbar^4|\partial_\sigma (e^{\Phi/\hbar}u)|^2\,\widehat{a}\, d\sigma d\tau\\
+\int_{\widehat {\mathcal V}_\rig} \left(-1-\kappa_{\max}\hbar^{2}+\hbar^2(\mathfrak v -|\Phi'|^2)-\tilde C\hbar^{3}-z\right)|e^{\Phi/\hbar}u|^2\,\widehat{a}\, d\sigma d\tau\,,
\end{multline*}
for some given constant $\tilde C>0\,$ independent of $R$.

Using the assumption on $z$, we deduce that
\begin{multline*}
\int_{\widehat {\mathcal V}_\rig}\widehat a^{-2}\hbar^4|\partial_\sigma (e^{\Phi/\hbar}u)|^2\,\widehat{a}\,d\sigma d\tau+\int_{\widehat {\mathcal V}_\rig} \left(\hbar^2(\mathfrak v -|\Phi'|^2)-\tilde C\hbar^{3}-\frac{M(R)}{2}\hbar^3\right)|e^{\Phi/\hbar}u|^2\,\widehat{a}\,d\sigma d\tau\\
\leq\|e^{\Phi/\hbar}u(\widehat{\mathcal{L}}_{\rig}-z)u\|\|e^{\Phi/\hbar}u\|\,.
\end{multline*}
Now we use the assumption on  the function $\Phi$ and obtain
\begin{multline*}
\int_{\widehat {\mathcal V}_\rig} \widehat a^{-2}\hbar^4|\partial_\sigma (e^{\Phi/\hbar}u)|^2\, \widehat a\, d\sigma d\tau+\left(\frac{M(R)}{2}-\tilde C\right)\int_{\widehat {\mathcal V}_\rig\setminus\hat{\mathcal B}_\rig(R\hbar^{1/2})}\hbar^3|e^{\Phi/\hbar}u|^2\, \widehat a\,  d\sigma d\tau\\
\leq \|e^{\Phi/\hbar}(\widehat{\mathcal{L}}_{\rig}-z)u\|\|e^{\Phi/\hbar}u\|+\tilde C\hbar^{3}\int_{\widehat{\mathcal  V}_{\rig}\cap\hat{\mathcal B}_\rig(R\hbar^{1/2})} |e^{\Phi/\hbar}u|^2\widehat a d\sigma d\tau \,.
\end{multline*}
We choose $R_{0}>0$ such that $\frac{M(R_{0})}{2}-\tilde C >0$. For all $R\geq R_{0}$, we have $\frac{M(R)}{2}-\tilde C\geq \frac{M(R_{0})}{2}-\tilde C>0$.

Thus, by the Cauchy-Schwarz inequality and the assumption on the function $\Phi$,
\begin{multline}\label{eq.estimate-weight3}
\int_{\widehat {\mathcal V}_\rig}\widehat a^{-2}\hbar^4|\partial_\sigma (e^{\Phi/\hbar}u)|^2\,\widehat{a}d\sigma d\tau+\frac{1}{2}\left(\frac{M(R_{0})}{2}-\tilde C\right)\hbar^3\|e^{\Phi/\hbar}u\|^2\\
\leq C\hbar^{-3}\|e^{\Phi/\hbar}(\widehat{\mathcal{L}}_{\rig}-z)u\|^2+C\hbar^3\|u\|_{L^2(\widehat{\mathcal  V}_{\rig}\cap\hat{\mathcal{B}}_{\rig}(R\hbar^{\frac{1}{2}}))}^2\,.
\end{multline}
 From \eqref{eq.estimate-weight3}, we get
\[\left(\frac{M(R_{0})}{2}-\tilde C\right)^{\frac{1}{2}}\hbar^{\frac{3}{2}}\|e^{\Phi/\hbar}u\|\leq C\hbar^{-\frac{3}{2}}\|e^{\Phi/\hbar}(\widehat{\mathcal{L}}_{\rig}-z)u\|+C\hbar^\frac{3}{2}\|u\|_{L^2(\widehat{\mathcal  V}_{\rig}\cap\hat{\mathcal{B}}_{\rig}(R\hbar^{\frac{1}{2}}))}\,,\]
and we deduce \eqref{eq.estimate-weight1}. The estimate \eqref{eq.estimate-weight2} directly comes from \eqref{eq.estimate-weight3}.
\end{proof}

\begin{remark}
If we apply Proposition \ref{prop.Agmon} to the eigenpair $(\phi_{\omega,\hbar}, \mu_{\rig}(\hbar))$, we  get
\[\|e^{\Phi/\hbar}\phi_{\omega, \hbar}\|_{L^2(\widehat{\mathcal  V}_{\rig})}\leq C\, \|\phi_{\omega, \hbar}\|_{L^2(\widehat{\mathcal  V}_{\rig})}\,,\]
(as soon as $M$ is large enough, to insure that $\mu_{\rig}(\hbar)$ belongs to the energy window). Note also that these estimates are weighted estimates in $H_{\sigma}^1(\omega, L_{\tau}^2(0,T  ))$.
\end{remark}

Let us gather some possible choices for $\Phi$ in the following proposition (see \cite[Chapter 6]{DS99} or \cite[Proposition 2.4 and Lemma 2.5]{BHR-c} for a detailed proof).

\begin{proposition}\label{prop.Agmon.choice}
Let $c_{0}>0$ such that
\begin{equation}\label{eq.c0}
\mathfrak v (\sigma)\geq c_{0}(\sigma-s_{\rig})^2\qquad \mbox{ and}\qquad \Phi_{\rig}(\sigma)\geq c_{0}(\sigma-s_{\rig})^2,\qquad\forall \sigma\in\omega\,.
\end{equation}
Possible choices of $\Phi$ satisfying the assumptions in Proposition~\ref{prop.Agmon}  are:
\begin{enumerate}[\rm (a)]
\item\label{poids1} for $\alpha \in(0,1)$, the rough weight $$ \Phi=\sqrt{1-\alpha }\, \Phi_{\rig}$$ with $R>0$ and $M=c_{0}\alpha R^2$\,;
\item\label{poids2} for $N\in\mathbb{N}^*$ and $\hbar\in(0,1)$, the accurate  weight $$\tilde{\Phi}_{\rig, N,\hbar} =\Phi_{\rig} - N \hbar \ln \left(\max \left( \frac{\Phi_{\rig}}{\hbar}, N\right)\right)\,,$$ with $R=\sqrt{\frac{N}{c_{0}}}$ and $M=N\inf_{\omega}\frac{\mathfrak v }{\Phi_{\rig}}$\,;
\item\label{poids3} for $\alpha \in(0,1)$, $\check \omega$ as above, $N\in\mathbb{N}^*$ and $\hbar\in(0,1)$, the intermediate weight
\begin{equation}\label{eq.poids3}
\hat{\Phi}_{\rig, \check \omega, N,\hbar} (\sigma) = \min \left\{\tilde{\Phi}_{\rig, N, \hbar} (\sigma), \sqrt{1-\alpha }\displaystyle{\inf_{t \in \supp\chi'_{{\check \omega}}} \left(\Phi_{\rig}(t) +\int_{[\sigma,t]}\sqrt{\mathfrak v (\tilde\sigma)}\, d\tilde\sigma\right)}\right\},
\end{equation}
with $R=\sqrt{\frac{N}{c_{0}}}$ and $M=N\min\left(\alpha ,\inf_{\omega}\frac{V}{\Phi_{\rig}}\right)$, where we recall that $\chi_{{\check \omega}}'$ is supported near $\partial\omega$ if $\check \omega$ is large enough.
\end{enumerate}

Moreover, the weight $\hat{\Phi}_{\rig, \check \omega, N,\hbar}$ satisfies the following. Let $K$ be a compact with $K\subset\{\chi_{\check \omega}=1\}$. For all $N\in\mathbb{N}^*$, there exists $\alpha_{0}$ such that for all $0<\alpha < \alpha_0$,  there exist $\hbar_{0}>0$ and $R>0$ such that, for all $\hbar\in(0,\hbar_{0})$, we have
\begin{enumerate}[\rm (i)]
\item\label{eq.lmhPhi1} $\hat{\Phi}_{\rig,\check \omega,  N,\hbar} \leq \Phi_{\rig}$ on $\omega$,
\item\label{eq.lmhPhi3} $\hat{\Phi}_{\rig, \check \omega, N,\hbar}=\tilde{\Phi}_{\rig, N,\hbar}$ on $K$,
\item\label{eq.lmhPhi2} $\hat{\Phi}_{\rig, \check \omega, N,\hbar} = \sqrt{1-\alpha }\, \Phi_{\rig}$ on $\supp \chi_{{\check \omega}}'$ (cf. Remark~\ref{cor.WKB}).
\end{enumerate}
\end{proposition}

\begin{remark}
We now assume  $\rho=\frac{1}{4}$ and  $D>\mathsf{S}$.
\end{remark}

\section{WKB approximation in the simple mini-well case}\label{sec.WKBapprox}
The aim of this section is to approximate the true eigenfunction $\phi_{\hbar,\rig}$ by the WKB function $\Psi_{\hbar,\rig}$ defined in \eqref{eq.WKB-quasimode}, or equivalently by $\psi_{\hbar, \check \omega}$ defined in \eqref{eq.WKB-quasimode'}.
\subsection{Main result} Let us introduce the orthogonal  projection on the space spanned by  $\phi_{\hbar,\rig}$:
\[\Pi_{\rig}\psi=\langle\psi,\phi_{\hbar,\rig}\rangle\phi_{\hbar,\rig}\,.\]

Recall that $\chi\in C_c^\infty([0,\infty))$ denotes a cutt-off function which is equal to $1$ near $0$.

\begin{proposition}\label{prop:WKB=gs}
Let $K$ be a compact set in $\omega$. There holds
\begin{align}
e^{\Phi_{\rig}/\hbar}(\Psi^\natural_{\hbar,\rig}-\Pi_{\rig}\Psi^\natural_{\hbar,\rig})=\,&\mathcal{O}(\hbar^\infty)\,,\label{eq:app.WKB1'}\\
e^{\Phi_{\rig}/\hbar}\partial_{\sigma}(\Psi^\natural_{\hbar,\rig}-\Pi_{\rig}\Psi^\natural_{\hbar,\rig})=\,& \mathcal{O}(\hbar^\infty)\,,\label{eq:app.WKB2'}
\end{align}
in $\mathcal C(K;L^2(0,T  ))$ and where we have let $\Psi^\natural_{\hbar, \omega}=\chi(T^{-1}\tau)\Psi_{\hbar, \omega}$.
\end{proposition}
We split the proof into four steps inspired by the presentation of \cite{BHR-c}. We choose $\check \omega$ so that $K\subset \{\chi_{\check \omega}=1\}$. Then, it is sufficient (see Remark \ref{cor.WKB}) to prove that
\begin{align}
e^{\Phi_{\rig}/\hbar}(\psi_{\hbar,{\check\omega}}-\Pi_{\rig}\psi_{\hbar,{\check \omega}})=\,&\mathcal{O}(\hbar^\infty)\,,\label{eq:app.WKB1}\\
e^{\Phi_{\rig}/\hbar}\partial_{\sigma}(\psi_{\hbar,{\check \omega}}-\Pi_{\rig}\psi_{\hbar,{\check \omega}})=\,& \mathcal{O}(\hbar^\infty)\,,\label{eq:app.WKB2}
\end{align}
in $\mathcal C(K;L^2(0,T  ))$.

\subsection{Estimating the $L^2$-norm}
Recall the definition of the domain $\widehat{\mathcal  V}_{\rig}$ in \eqref{eq:dom-right}. Let
\begin{equation}\label{eq:choice.u}
u=\psi_{\hbar,{\check\omega}}-\Pi_\rig\psi_{\hbar,{\check\omega}}\,.
\end{equation}
Since $u$ is orthogonal to the eigenfunction $\phi_{\hbar,\rig}$, then by the min-max principle,
\[\big(\mu_{2,\rig}(\hbar)-\mu_{\rig}(\hbar)\big)\|u\|^2_{L^2(\widehat{\mathcal  V}_{\rig})}\leq \left\|\big(\widehat{\mathcal L}_\rig-\mu_{\rig}(\hbar)\big)u\right\|^2_{L^2(\widehat{\mathcal  V}_{\rig})}=\left\|\big(\widehat{\mathcal L}_\rig-\mu_{\rig}(\hbar)\big)\psi_{\hbar,{\check\omega}}\right\|^2_{L^2(\widehat{\mathcal  V}_{\rig})}\,.\]
Using the estimate of the $\mu_{2,\rig}(\hbar)-\mu_\rig(\hbar)$ in \eqref{eq:sp-gap.1well}, the expansion of $\mu_\rig(\hbar)$ in \eqref{eq:Tay-mu} and the result in Proposition~\ref{prop.WKB}, we get
\begin{equation}\label{eq:est.WKB}
\|u\|_{L^2(\widehat{\mathcal  V}_{\rig})}=\mathcal O(\hbar^\infty)\,.
\end{equation}

\subsection{Estimating $\big(\widehat{\mathcal L}_\rig-\mu_{\rig}(\hbar)\big)\psi_{\hbar,{\check\omega}}$}
Here we will prove that,
\begin{equation}\label{eq:est.WKB'}
\Big\|e^{\hat\Phi_{\rig,\check\omega, N,\hbar}/\hbar}\big(\widehat{\mathcal L}_\rig-\mu_{\rig}(\hbar)\big)\psi_{\hbar,{\check \omega}}\Big\|_{L^2(\widehat{\mathcal  V}_{\rig})}=\mathcal O(\hbar^\infty)\,.
\end{equation}
In view of the definition of the function $\psi_{\hbar,{\check\omega}}$ in Proposition~\ref{prop.WKB}, we write,
\begin{equation}\label{eq:est.WKB''}
\begin{aligned}
&e^{\hat\Phi_{\rig,\check\omega, N,\hbar}/\hbar}\big(\widehat{\mathcal L}_\rig-\mu_{\rig}(\hbar)\big)\psi_{\hbar,{\check\omega}}
\\
&\qquad=e^{\hat\Phi_{\rig,\check \omega, N,\hbar}/\hbar}\chi_{{\check \omega}}(\sigma)\chi(T^{-1}\tau)\big(\widehat{\mathcal L}_\rig-\mu_{\rig}(\hbar)\big)\Psi_{\hbar,\rig}\\
&\qquad+e^{\hat\Phi_{\rig,\check\omega, N,\hbar}}\Big(\chi_{{\check \omega}}(\sigma)[\widehat{\mathcal L}_\rig,\chi(T^{-1}\tau)]+\chi(T^{-1}\tau)[\widehat{\mathcal L}_\rig,\chi_{\check\omega}(\sigma)]\Big)\Psi_{\hbar,\rig}\,,
\end{aligned}
\end{equation}
where $[\cdot,\cdot]$ denotes the commutator.

Then we have
\begin{multline}\label{eq:est.WKB''-bis}
e^{\hat\Phi_{\rig,\check \omega, N,\hbar}/\hbar}\big(\widehat{\mathcal L}_\rig-\mu_{\rig}(\hbar)\big)\psi_{\hbar,{\check \omega}}=e^{(\hat\Phi_{\rig,\check \omega,N,\hbar}-\Phi_{\omega})/\hbar}\,\mathcal{O}(\hbar^\infty)\\
+e^{(\hat\Phi_{\rig,\check\omega,N,\hbar}-\Phi_{\omega})/\hbar}\Big(\chi_{{\check\omega}}(\sigma)[\widehat{\mathcal L}_\rig,\chi(T^{-1}\tau)]\Big)e^{\Phi_{\omega}/\hbar}\Psi_{\hbar,\rig}+e^{(\hat\Phi_{\rig,\check \omega,N,\hbar}-\Phi_{\omega})/\hbar}{\mathcal{O}_{L^\infty(\supp\chi'_{{\check \omega}})}(1)}\,.
\end{multline}
Here the notation $\mathcal{O}_{L^\infty(\supp\chi'_{{\check \omega}})}(1)$ means that the function is supported on $\supp\chi'_{\check\omega}$ and that it is uniformly bounded when $\hbar$ goes to $0$.
By Proposition~\ref{prop.Agmon.choice}, $0<e^{(\hat\Phi_{\rig,\check\omega,N,\hbar}-\Phi_{\rig})/\hbar}\leq 1$ in $\omega$ and
for $\alpha\in(0,1)$,
\[\hat\Phi_{\rig,\check\omega,N,\hbar}-\Phi_{\rig}\leq -(1-\sqrt{1-\alpha})\Phi_\rig\]
in ${\rm supp}\,\chi'_{{\check \omega}}$. Now, \eqref{eq:est.WKB'} becomes  a consequence of \eqref{eq:est.WKB''-bis} and Proposition~\ref{prop.WKB} thanks to support considerations.

\subsection{Proof of \eqref{eq:app.WKB1}}
Let us apply Proposition~\ref{prop.Agmon} with the following choices:
$u$ as in \eqref{eq:choice.u}, $z=\mu_\rig(\hbar)$ and $\Phi=\hat\Phi_{\rig,\check\omega,N,\hbar}$. We have
\begin{multline*}
\|e^{\hat\Phi_{\rig,\check\omega, N,\hbar}/\hbar}u\|_{L^2(\widehat{\mathcal V}_{\rig})}+
\|\partial_\sigma e^{\hat\Phi_{\rig,\check\omega, N,\hbar}/\hbar}u\|_{L^2(\widehat{\mathcal V}_{\rig})}
\\
\leq C\hbar^{-7}\| e^{\hat\Phi_{\rig,\check \omega, N,\hbar}/\hbar}(\widehat{\mathcal L}_\rig -\mu_\rig(\hbar)) u\|_{L^2(\widehat{\mathcal V}_{\rig})}+C\hbar^{-1}\|u\|_{L^2(\widehat{\mathcal V}_{\rig})}\,.
\end{multline*}
In light of \eqref{eq:est.WKB} and \eqref{eq:est.WKB'}, we deduce that,
\begin{equation}\label{eq:est.WKB3}
\|e^{\hat\Phi_{\rig,\check \omega, N,\hbar}/\hbar}u\|_{L^2(\widehat{\mathcal V}_{\rig})}+
\|\partial_\sigma e^{\hat\Phi_{\rig,\check \omega, N,\hbar}/\hbar}u\|_{L^2(\widehat{\mathcal V}_{\rig})}
=\mathcal O(\hbar^\infty)\,.
\end{equation}
By Proposition~\ref{prop.Agmon.choice}, we have
$$\hat\Phi_{\rig,\check \omega,N,\hbar}=\tilde\Phi_{\rig,N,\hbar}{~\rm in~}K\,,\quad{\rm and}\quad e^{(\Phi_{\rig}-\tilde\Phi_{\rig,N,\hbar})/\hbar}=\mathcal O(\hbar^{-N}){\rm ~in~}L^\infty(K)\,.$$
In that way, we get the following estimate,
\begin{equation}\label{eq:est.WKB3'}
\|e^{\Phi_{\rig}/\hbar}u\|_{L^2(\hat K)}+
\|\partial_\sigma e^{\Phi_{\rig}/\hbar}u\|_{L^2(\hat K)}
=\mathcal O(\hbar^\infty)\,,\end{equation}
where
$$\hat K=K\times (0,T  )\,.$$
We may rewrite this estimate in the form,
\begin{equation}\label{eq. exp-u-approx}
\|e^{\Phi_{\rig}/\hbar}u\|_{L^2(K;L^2(0,T  ))}+
\|\partial_\sigma e^{\Phi_{\rig}/\hbar}u\|_{L^2(K;L^2(0,T  ))}
=\mathcal O(\hbar^\infty)\,,
\end{equation}
which in turn yields \eqref{eq:app.WKB1} in $\mathcal{C}(K; L^2(0,T  ))$ (cf. \cite[Thm.~2; p.~302]{Ev}).

\subsection{Proof of \eqref{eq:app.WKB2}}~\\
Let (cf. \eqref{eq:choice.u})
) $$v:= \partial_\tau u= \partial_\tau(\psi_{\hbar,{\check\omega}}-\Pi_{\rig}\psi_{\hbar,{\check \omega}})$$ and
$$w:=\partial_\tau v=\partial_\tau^2 u\,.$$

We apply Proposition~\ref{prop.Agmon} to obtain,
\begin{equation}\label{eq:est.WKB4}
\|e^{\hat\Phi_{\rig,\check \omega, N,\hbar}/\hbar}v\|\leq C\hbar^{-3}\|e^{\hat\Phi_{\rig,\check \omega, N,\hbar}/\hbar}\big(\mathcal L_{\rig}-\mu_\rig(\hbar)\big)v\|+C\|v\|\end{equation}
 and
 \begin{equation}
\|e^{\hat\Phi_{\rig,\check \omega, N,\hbar}/\hbar}w\|\leq C\hbar^{-3}\|e^{\hat\Phi_{\rig,\check \omega, N,\hbar}/\hbar}\big(\mathcal L_{\rig}-\mu_\rig(\hbar)\big)w\|+C\|w\|\,.
\end{equation}
In light of the two identities
\begin{equation*}
\big(\mathcal L_{\rig}-\mu_\rig(\hbar)\big)v=\partial_\tau\big(\mathcal L_{\rig}-\mu_\rig(\hbar)\big)u=\partial_\tau\big(\mathcal L_{\rig}-\mu_\rig(\hbar)\big)\psi_{\hbar,{\check \omega}}
\end{equation*}
and
\begin{equation*}
\big(\mathcal L_{\rig}-\mu_\rig(\hbar)\big)v=\partial_\tau^2\big(\mathcal L_{\rig}-\mu_\rig(\hbar)\big)\psi_{\hbar,{\check \omega}}\,,
\end{equation*}
we get by Proposition~\ref{prop.WKB},
\begin{equation}\label{eq:est**}
e^{\Phi_\rig/\hbar}\big(\mathcal L_{\rig}-\mu_\rig(\hbar)\big)v=\mathcal O(\hbar^\infty)\quad {\rm and}
\quad e^{\Phi_\rig/\hbar}\big(\mathcal L_{\rig}-\mu_\rig(\hbar)\big)w=\mathcal O(\hbar^\infty)\,,\end{equation}
in $L^2(\widehat{\mathcal V}_{\rig})$.\\
By Step~1, Proposition~\ref{prop.WKB} and Remark \ref{cor.WKB},
$$\big(\widehat{\mathcal L}_\rig -\mu_\rig(\hbar)\big)u=\mathcal O(\hbar^\infty)\quad{\rm in}~L^2(\widehat{\mathcal V}_{\rig}) \,.$$
By \eqref{eq:est.WKB}, we get further
$$\widehat{\mathcal L}_\rig u=\mathcal O(\hbar^\infty)\quad{\rm in}~L^2(\widehat{\mathcal V}_{\rig}) \,.$$
Multiplying by $u$ and integrating by parts yields,
\[\int_{\widehat{\mathcal V}_{\rig}} \big(|\partial_\tau u|^2+\hbar^4\widehat a^{-2}|\partial_\sigma u|^2\big)\widehat a\,d\tau d\sigma-\int_{\omega} |u(\sigma,0)|^2\,d\sigma=\mathcal O(\hbar^\infty)\,.\]
Note that the lowest eigenvalue of the operator $-\partial_\tau^2$ in $L^2(\R_+)$ with boundary condition \break  $w' (0)=-2w(0)$ is equal to $-4$. Then we get
$$\widehat a\geq \frac12\quad{\rm and}\quad \int_{\widehat{\mathcal V}_{\rig}} |\partial_\tau u|^2d\tau d\sigma-2\int_{\omega} |u(\sigma,0)|^2\,d\sigma\geq -4\int_{\widehat{\mathcal V}_{\rig}} |u|^2d\tau d\sigma\,,$$
and we deduce
$$
\frac12\int_{\widehat{\mathcal V}_{\rig}} |\partial_\tau u|^2 \widehat{a}\,d\tau d\sigma\leq \mathcal O(\hbar^\infty)+4\|u\|^2_{L^2(\widehat{\mathcal V}_{\rig})}=\mathcal O(\hbar^\infty)\,.
$$
In a similar way, using $\big(\mathcal L_{\rig}-\mu_\rig(\hbar)\big)v=\mathcal O(h^\infty)$ (cf. \eqref{eq:est**}), we get,
$$
\frac12\int_{\widehat{\mathcal V}_{\rig}} |\partial_\tau v|^2\,  \widehat{a}\,d\tau d\sigma\leq \mathcal O(\hbar^\infty)+4\|v\|_{L^2(\widehat{\mathcal V}_{\rig})}^2\,.
$$
Thus, we have the following two important estimates:
$$\|v\|=\mathcal O(\hbar^\infty)\quad{\rm and}\quad \|w\|=\mathcal O(\hbar^\infty)\,.$$
This and the estimates in \eqref{eq:est**} allow us to repeat the argument in Step~2 to obtain
$$\|e^{\hat\Phi_{\rig,\check \omega, N,\hbar}/\hbar}\big(\mathcal L_{\rig}-\mu_\rig(\hbar)\big)v\|=\mathcal O(\hbar^\infty)\quad{\rm and}\quad \|e^{\hat\Phi_{\rig,\check \omega, N,\hbar}/\hbar}\big(\mathcal L_{\rig}-\mu_\rig(\hbar)\big)w\|=\mathcal O(\hbar^\infty)\,.$$
Now,  \eqref{eq:est.WKB4} yields
$$
\|e^{\hat\Phi_{\rig,\check \omega, N,\hbar}/\hbar}\partial_\tau u\| =\mathcal O(\hbar^\infty) \quad{\rm and}\quad \|e^{\hat\Phi_{\rig,\check \omega, N,\hbar}/\hbar}\partial_\tau^2 u\| =\mathcal O(\hbar^\infty)\,.
$$
As done in Step~3, the properties of the function $\hat\Phi_{\rig,\hat \omega,N,\hbar}$ yield
\begin{equation}\label{eq:est.WKB5}
\|e^{\Phi_{\rig}/\hbar}\partial_\tau u\| =\mathcal O(\hbar^\infty) \quad{\rm and}\quad \|e^{\Phi_{\rig}/\hbar}\partial_\tau^2 u\| =\mathcal O(\hbar^\infty)\quad{\rm in ~}L^2(\widehat K)\,.
\end{equation}
Since $|\Phi'_{ \omega}|$ is bounded independently of $\hbar$, then we infer from \eqref{eq:est.WKB3'}
\begin{equation}\label{eq:est.WKB3''}
e^{\Phi_{\rig}/\hbar}\partial_\sigma u=\mathcal O(\hbar^\infty)\quad{\rm in~}L^2(\widehat K)\,.
\end{equation}
We recall that we have
\[e^{\Phi_{\rig}/\hbar}\big(\mathcal L_{\rig}-\mu_\rig(\hbar)\big)u=\mathcal O(\hbar^\infty)\,,\quad{\rm in~}L^2(\widehat K)\,.\]
Then, we use the estimates \eqref{eq:est.WKB5}, \eqref{eq:est.WKB3''} and \ref{eq. exp-u-approx} to deduce
\[e^{\Phi_{\rig}/\hbar}\partial_\sigma^2u=\mathcal O(\hbar^\infty)\quad{\rm in~}L^2(\widehat K)\,.\]
Since $|\Phi'_{\omega}|=\mathcal O(1)$, we get further
$$\partial_\sigma\Big(e^{\Phi_{\rig}/\hbar}\partial_\sigma u\Big)=\mathcal O(\hbar^\infty)\quad{\rm in~}L^2(\widehat K)\,.$$
This estimate and \eqref{eq:est.WKB3'} yield the estimate in \eqref{eq:app.WKB2}.

\section{Double mini-wells and interaction matrix}\label{sec.doublewell}
In this section, we come back to the study of the double mini-wells operator $\widehat{\mathcal{L}}_{\hbar}$.

\subsection{Right and left operators}
We introduce the operators corresponding to the left and right wells. Recall that we identify the boundary $\Gamma=\Gamma$ with the interval $(-L,L]$ and the orientation is chosen counter-clockwise. By Assumption~\ref{hyp.main3}, we know that
\[\{\sigma\in(-L,L)~:~\kappa(\sigma)=\kappa_{\max}\}=\{-s_0,s_0\}\]
where $s_0\in(0,L)$. We introduce the two (mini-)wells
\begin{equation}\label{eq:s-lr}
s_{\mathsf{\ell}}=s_0\quad{\rm and}\quad s_{\mathsf{r}}=-s_0\,,
\end{equation}
and this is consistent with the counter-clockwise orientation of the boundary.

Let us consider $ \eta$ such that
\[0<\eta<\frac12\min\Big(s_{\mathsf{\ell}},L-s_{\mathsf{\ell}}\Big)=\frac12\min\Big(-s_{\mathsf{r}},-L+s_{\mathsf{r}}\Big)\,.\]
We introduce the two intervals (in $\Gamma$)
\begin{equation}\label{eq:lr-intervals}
\omega_\mathsf{\ell}=\{\sigma\in(-L,L]~:~|\sigma-s_{\mathsf{r}}|>\eta\}\quad{\rm and}\quad
\omega_\mathsf{r}=\{\sigma\in(-L,L]~:~|\sigma-s_{\mathsf{\ell}}|>\eta\}\,.
\end{equation}
We will apply the results in Section \ref{sec.miniwell} with the interval $\omega$  being  $\omega_\mathsf{\ell}$ or $\omega_\mathsf{r}$. The assumption on $\eta$ ensures that the  left  and right intervals $\omega_{\mathsf{\ell}}$ and $\omega_{\mathsf{r}}$  have the same length $2L-2\eta$ and are related by the simple transformation $\sigma\mapsto-\sigma$.

Let us introduce the two sets
\[\hat {\omega}_\mathsf{\ell}=\omega_\mathsf{\ell}\times(0,T)\,,\quad \hat{\omega}_\mathsf{r}=\omega_\mathsf{r}\times(0,T)\,,\]
and the unitary transform $U$ defined by
\[U f (\sigma,\tau)=f(-\sigma,\tau)\,.\]
This transform goes from $L^2(\hat {\omega}_\mathsf{\ell})$ to $L^2(\hat{\omega}_\mathsf{r})$. Due to the symmetry assumption (cf. Assumption~\ref{hyp.main3}), we notice that
\[\widehat{\mathcal{L}}_{\lef}=U\widehat{\mathcal{L}}_{\righ} U^{-1}\,.\]
Thus these operators have the same spectrum and we may denote by $\mu(\hbar)$ their common lowest eigenvalue, i.e.
\begin{equation}\label{eq:ev.lef=rig}
\mu(\hbar)=\mu_\lef(\hbar)=\mu_\righ(\hbar)\,.
\end{equation}
The eigenfunctions of $\widehat{\mathcal{L}}_{\lef}$ may  be deduced from the ones of $\widehat{\mathcal{L}}_{\righ}$. In particular,
\begin{equation}\label{eq:gs.lef=rig}
\phi_{\hbar,\lef}=U\phi_{\hbar,\righ}\,.
\end{equation}
In Remark~\ref{cor.WKB}, we choose $\check \omega$ as
$$\check \omega_\mathsf{\ell}=\{\sigma~:~|\sigma-s_{\righ}|>2\eta\}\quad{\rm or}\quad \check \omega_\mathsf{r}=\{\sigma~:~|\sigma-s_\lef|>2\eta\}\,.$$
As a consequence, we get the two cut-off functions
$$\chi_{\lef}=\chi_{{\check \omega}_\lef}\quad{\rm and}\quad \chi_{\righ}=\chi_{{\check \omega}_\righ}$$
that are equal to $1$ in $\check \omega_\mathsf{\ell}$ and $\check\omega_\mathsf{\righ}$ respectively.

Proposition~\ref{prop:WKB=gs} yields, for every compact set $K\subset \check \omega_{\lef}$,
\begin{align}
e^{\Phi_{\lef}/\hbar}(\Psi^\natural_{\hbar,\lef}-\Pi_{\lef}\Psi^\natural_{\hbar,\lef})=\,&\mathcal{O}(\hbar^\infty)\,,\label{eq:app.WKB1-lef}\\
e^{\Phi_{\lef}/\hbar}\partial_{\sigma}(\Psi^\natural_{\hbar,\lef}-\Pi_{\lef}\Psi^\natural_{\hbar,\lef})=\,& \mathcal{O}(\hbar^\infty)\,,\label{eq:app.WKB2-lef}
\end{align}
in $\mathcal{C}(K;L^2(0,T  ))$. Here
\[
\Pi_\lef\psi=\langle\psi,\phi_{\hbar,\lef}\rangle\phi_{\hbar,\lef}\quad{\rm and}\quad \Phi_\lef(\sigma)=\Phi_\righ(-\sigma)=\int_{[s_{\lef},\sigma]}\sqrt{\kappa_{\max}-\kappa(\tilde\sigma)}\,d\tilde\sigma\,.
\]

\subsection{Estimates of Agmon}
We introduce the global weight
\[\Phi=\min\left(\Phi_{\righ},\Phi_{\lef}\right)\,,\]
where
\[\Phi_{\righ}(\sigma)=\int_{[s_{\righ},\sigma]} \sqrt{\mathfrak{v}(\tilde\sigma)} \, d\tilde\sigma\,,\quad\forall\sigma\in\omega_{\righ}\,,\qquad \Phi_{\lef}(\sigma)=\int_{[s_{\lef},\sigma]} \sqrt{\mathfrak{v}(\tilde\sigma)} \, d\tilde\sigma\,,\quad\forall\sigma\in\omega_{\lef}\,.\]
We stress one more time that the integration over the segment $[\sigma_1,\sigma_2]$ means the line integral along the boundary $\Gamma$ from the point $\sigma_1$ to the point $\sigma_2$ in the counterclockwise direction, see Figure~2.
\begin{figure}\label{fig2}
\begin{center}
\includegraphics[width=10cm]{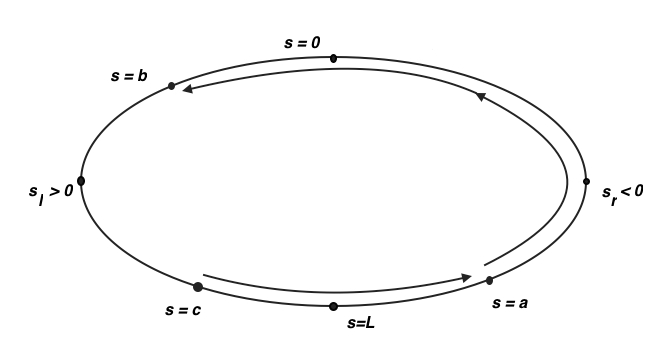} 
\end{center}
\caption{Illustration of the integral over a segment. In this case, we have $\int_{[a,b]}=\int_a^b$ and $\int_{[c,a]}=\int_c^L+\int_{-L}^a$.}
\end{figure}

In particular, we have
$$ |\Phi'(\sigma)|^2=\mathfrak v(\sigma) \,.$$
 Let us define
\begin{equation}\label{eq:def.S}
\mathsf{S}_{\mathsf{u}}=\Phi_\righ(s_\lef)\,,\quad  \mathsf{S}_{\mathsf{d}}=\Phi_\lef(s_\righ)\quad{\rm and}\quad \mathsf{S}=\min(\mathsf{S}_{\mathsf{u}},\mathsf{S}_{\mathsf{d}})\,.
\end{equation}
Note that, on the \enquote{upper part}, $\Phi_{\righ}+\Phi_{\lef}=\mathsf{S}_{\mathsf{u}}$ and on the \enquote{lower part} $\Phi_{\righ}+\Phi_{\lef}=\mathsf{S}_{\mathsf{d}}$. In particular, we have
\begin{equation}
\Phi_{\righ}+\Phi_{\lef}\geq\mathsf{S}\,.
\end{equation}

The following proposition may be established by using the same estimates as in the proof of Proposition \ref{prop.Agmon}.
\begin{proposition}
For all $\alpha\in(0,1)$, for all $C_{0}>0$, there exist positive constants $\hbar_{0}$, $A$, $c$, $C$ such that, for all $\hbar\in(0,\hbar_{0})$, $z\in[-1-\kappa_{\max}\hbar^2, -1-\kappa_{\max}\hbar^2+C_{0}\hbar^3]$, $u\in\widehat{\mathcal{D}}$,
\[c\hbar^3\|e^{\sqrt{1-\alpha}\Phi/\hbar}u\|_{L^2}\leq \|e^{\sqrt{1-\alpha}\Phi/\hbar}\left(\widehat{\mathcal{L}}_{\hbar}-z\right)u\|_{L^2}+C\hbar^3\|u\|_{L^2(\widehat{\mathcal{B}}(A\hbar^{\frac{1}{2}}))}\,,\]
and
\[\hbar^4\|\partial_{\sigma}(e^{\sqrt{1-\alpha}\Phi/\hbar}u)\|^2_{L^2}\leq C\hbar^{-3}\|e^{\sqrt{1-\alpha}\Phi/\hbar}\left(\widehat{\mathcal{L}}_{\hbar}-z\right)u\|^2_{L^2}+C\hbar^3\|u\|^2_{L^2(\widehat{\mathcal{B}}(A\hbar^{\frac{1}{2}}))}\,,\]
where $\widehat{\mathcal{B}}(\varrho)=\widehat{\mathcal{B}}_{\righ}(\varrho)\cup\widehat{\mathcal{B}}_{\lef}(\varrho)$, where $\widehat{\mathcal{B}}_r := \widehat{\mathcal{B}}_{\omega_r}$, resp.  $\widehat{\mathcal{B}}_\ell  := \widehat{\mathcal{B}}_{\omega_\ell}$ (cf Notation \ref{not.Bhat}).
\end{proposition}

\subsection{Interaction matrix}
\subsubsection{Preliminary considerations}

\begin{definition}
Let us introduce the two quasimodes
\[f_{\hbar,\righ}=\chi_{\righ}\phi_{\hbar,\righ}\,,\qquad f_{\hbar,\lef}=\chi_{\lef}\phi_{\hbar,\lef}\,,\]
that clearly belong to the domain of $\widehat{\mathcal{L}}_{h}$.
\end{definition}
We use the following convenient notation.
\begin{notation}
For $M>0$, the notation $\tilde{\mathcal O}(e^{-M/h})$ (introduced by Helffer-Sj\"ostrand in \cite{HSj}) stands for a quantity $r(h,\eta)$ defined on a set of the form $(0,h_0)\times (0,\eta_0)$ and satisfying the following: There exists a function $\gamma:(0,\infty)\to\R$ such that $\lim_{\eta\to0}\gamma(\eta)=0$, and for all $\varepsilon>0$ and $\eta>0$,
$r(h,\eta)=\mathcal O(e^{(\varepsilon+\gamma(\eta)-M)/h})$. The parameter $\eta$ will measure the distance between $\complement\omega_{\alpha}$ and $\check \omega_{\alpha}$, for $\alpha\in\{\righ,\lef\}$.
\end{notation}
The following lemma is the consequence of Agmon's estimates  and considerations on the supports.
\begin{lemma}\label{lem.prelim}
For $\alpha\in\{\righ,\lef\}$, we let
\[r_{\hbar, \alpha}=(\widehat{\mathcal{L}}_{\hbar}-\mu(\hbar))f_{\hbar,\alpha}=[\widehat{\mathcal{L}}_{\righ},\chi_{\alpha}]\phi_{\hbar,\alpha}\,.\]
Then, we have
\begin{enumerate}[\rm (i)]
\item $r_{\hbar, \alpha}=\tilde{\mathcal{O}}(e^{-\mathsf{S}/\hbar})$,
\item $\langle r_{\hbar,\alpha},f_{\hbar, \alpha}\rangle=\tilde{\mathcal{O}}(e^{-2\mathsf{S}/\hbar})$ and $\langle r_{\hbar,\alpha},f_{\hbar, \beta}\rangle=\tilde{\mathcal{O}}(e^{-\mathsf{S}/\hbar})$ for $\alpha\neq\beta$,
\item $\langle f_{\hbar,\alpha}, f_{\hbar,\alpha}\rangle=1+\tilde{\mathcal{O}}(e^{-2\mathsf{S}/\hbar})$ and $\langle f_{\hbar,\alpha}, f_{\hbar,\beta}\rangle=\tilde{\mathcal{O}}(e^{-\mathsf{S}/\hbar})$ for $\alpha\neq\beta$,
\item If $\mathcal{F}=\mathsf{ span }\{f_{\hbar,\righ}\,,f_{\hbar,\lef}\}$, then $\dim\mathcal{F}=2$.
\end{enumerate}
\end{lemma}

The following lemma states that the first two eigenvalues of $\widehat{\mathcal{L}}_{\hbar}$ are close to $\mu(\hbar)$ (the common first eigenvalue of the two mini-well operators) modulo $\tilde{\mathcal{O}}(e^{-\mathsf{S}/\hbar})$. The proof is standard (see \cite{DS99} or the presentation in \cite{BHR-c}).

\begin{lemma}\label{lem.expSh}
Let us define $\mathcal{G}=\mathsf{range}\,\mathds{1}_{I_{\hbar}}(\widehat{\mathcal{L}}_{\hbar})$ where $I_{\hbar}=(-\infty, -1-\kappa_{\max}\hbar^2+2\gamma\hbar^3)$. Then we have
\begin{enumerate}[\rm (i)]
\item $\mathrm{dist} (\mathrm{sp}(\widehat{\mathcal{L}}_{\hbar}),\mu(\hbar))=\tilde{\mathcal{O}}(e^{-\mathsf{S}/\hbar})$,
\item $\langle (\widehat{\mathcal{L}}_{\hbar}-\mu(\hbar)) u, u\rangle\geq \gamma\,\hbar^3\|u\|^2\,, \forall u\in\mathcal{G}^\perp$,
\item $\dim\mathcal{G}=2$,
\item $\mathrm{sp}(\widehat{\mathcal{L}}_{\hbar})\cap I_{\hbar}\subset [\mu(\hbar)-\tilde{\mathcal{O}}(e^{-\mathsf{S}/\hbar}),\mu(\hbar)+\tilde{\mathcal{O}}(e^{-\mathsf{S}/\hbar})]$.
\end{enumerate}
\end{lemma}

\subsubsection{Interaction matrix}
We want a more accurate description of the splitting between the first two eigenvalues of $\widehat{\mathcal{L}}_{\hbar}$. For that purpose, we will consider the restriction of $\widehat{\mathcal{L}}_{\hbar}$ to the space  $\mathcal{G}$ generated by the first two eigenfunctions and we will exhibit an orthonormal basis of this space that allows us to compute asymptotically the eigenvalues of the corresponding $2\times 2$ matrix.

Let us introduce $\Pi$ the orthogonal projection on $\mathcal{G}$ and $g_{\hbar,\alpha}=\Pi f_{\hbar, \alpha}$. As a consequence of Lemma \ref{lem.expSh} and of the spectral theorem, we get the following lemma.
\begin{lemma}
We have in $H^1$,
\[f_{\hbar,\alpha}-g_{\hbar,\alpha}=\tilde{\mathcal{O}}(e^{-\mathsf{S}/\hbar})\,.\]
\end{lemma}
From this lemma and Lemma \ref{lem.prelim}, we deduce the following.
\begin{lemma}
Let us define the $2\times 2$ matrix $\mathsf{T}$ by $\mathsf{T}_{\alpha,\beta}=\langle f_{\hbar,\alpha}, f_{\hbar, \beta}\rangle$ for $\alpha\neq\beta$ and $0$ otherwise. Then, we have
\begin{enumerate}[\rm (i)]
\item $\mathsf{T}=\tilde{\mathcal{O}}(e^{-\mathsf{S}/\hbar})$,
\item $(\langle f_{\hbar,\alpha}, f_{\hbar,\beta}\rangle)_{\alpha,\beta}=\mathsf{Id}+\mathsf{T}+\tilde{\mathcal{O}}(e^{-2\mathsf{S}/\hbar})$,
\item $\langle g_{\hbar,\alpha}, g_{\hbar,\beta}\rangle=\langle f_{\hbar,\alpha}, f_{\hbar,\beta}\rangle+\tilde{\mathcal{O}}(e^{-2\mathsf{S}/\hbar})$,
\item $(\langle g_{\hbar,\alpha}, g_{\hbar,\beta}\rangle)_{\alpha,\beta}=\mathsf{Id}+\mathsf{T}+\tilde{\mathcal{O}}(e^{-2\mathsf{S}/\hbar})$.
\end{enumerate}
\end{lemma}
Let us now examine the so-called interaction matrix. The family $(g_{\hbar,\alpha})$ generates $\mathcal{G}$ but is a priori not orthonormal. Thus we use the Gram-Schmidt matrix $\mathsf{G}=(\langle g_{\hbar,\alpha}, g_{\hbar,\beta}\rangle)_{\alpha,\beta}$ and we let $\mathsf{g}=g\mathsf{G}^{-\frac{1}{2}}$ where $g$ is the row vector $(g_{\hbar,\lef}, g_{\hbar, \righ})$. The family $\mathsf{g}$ is now an orthonormal basis of $\mathcal{G}$. Let $\mathsf{M}$ be the (interaction) matrix of (the restriction of) $\widehat{\mathcal{L}}_{\hbar}$ in the basis $\mathsf{g}$.
\begin{proposition}\label{prop.splittingLhat}
We have
\[\mathsf{M}=\mu(\hbar)\mathsf{Id}+\mathsf{W}+\tilde{\mathcal{O}}(e^{-2\mathsf{S}/\hbar})\,,\]
where $\mathsf{W}$ is defined by $w_{\alpha,\beta}=\langle r_{\hbar,\alpha}, f_{\hbar, \beta}\rangle$ if $\alpha\neq\beta$ and $0$ otherwise. Moreover $\mathsf{W}$ is symmetric.
In particular, the splitting between the first two eigenvalues of $\widehat{\mathcal{L}}_{\hbar}$ is given by
\[\hat\lambda_{2}(\hbar)-\hat\lambda_{1}(\hbar)=2|w_{\lef,\righ}(\hbar)|+\tilde{\mathcal{O}}(e^{-2\mathsf{S}/\hbar})\,.\]
\end{proposition}
\subsection{Computation of the interaction}
Now the problem is to estimate the interaction term $w_{\lef,\righ}(\hbar)$ given by
\[w_{\lef,\righ}(\hbar)=\langle(\widehat{\mathcal{L}}_{\hbar}-\mu(\hbar))f_{\hbar,\lef},f_{\hbar,\righ} \rangle=\langle[\widehat{\mathcal{L}}_{\hbar},\chi_{\lef}]\phi_{\hbar,\lef},\chi_{\righ}\phi_{\hbar,\righ} \rangle\,.\]
Let us recall that
\[\widehat{\mathcal L}_{\righ}=-\hbar^4\,\widehat a^{-1} (\partial_\sigma\widehat a^{-1})\partial_\sigma-\widehat a^{-1}(\partial_\tau\widehat a)\partial_\tau\,.\]
Since $\chi_{\lef}$ does not depend on $\tau$, we get
\[w_{\lef,\righ}(\hbar)=-\hbar^4\langle[\widehat a^{-1} (\partial_\sigma\widehat a^{-1})\partial_\sigma,\chi_{\lef}]\phi_{\hbar,\lef},\chi_{\righ}\phi_{\hbar,\righ} \rangle\,,\]
where
$$\chi_\lef(\sigma)=\chi_\righ(-\sigma)\,.$$
After the computation of the commutator and an integration by parts (with respect to $\sigma$) to eliminate $\chi''_{\lef}$, we get
$$
w_{\lef,\righ}(\hbar)=\hbar^4\int_{\widehat{\mathcal V}_{\rho}} \widehat{a}^{-1}\chi_{\righ}\, \chi'_{\lef}\left((\partial_{\sigma}\phi_{\hbar, \righ})\phi_{\hbar, \lef}-\phi_{\hbar, \righ}(\partial_{\sigma}\phi_{\hbar, \lef})\right)d\sigma d\tau\,.
$$
Since $\chi_{\righ}=1$ in the support of $\chi'_\lef$, we get,
\[ w_{\lef,\righ}(\hbar)=\hbar^4\int_{\widehat{\mathcal V}_{\rho}} \widehat{a}^{-1}\chi_{\lef}'\left((\partial_{\sigma}\phi_{\hbar, \righ})\phi_{\hbar, \lef}-\phi_{\hbar, \righ}(\partial_{\sigma}\phi_{\hbar, \lef})\right)d\sigma d\tau\,.\]
Then, we integrate by parts and use the fact that $\phi_{\hbar,\alpha}$ is an eigenfunction of $\widehat{\mathcal{L}}_{\alpha}$ to get
\[w_{\lef,\righ}(\hbar)=w^{\mathsf{u}}_{\lef,\righ}(\hbar)+w^{\mathsf{d}}_{\lef,\righ}(\hbar)\,,\]
where
\begin{align*}
w^{\mathsf{u}}_{\lef,\righ}(\hbar)&=\hbar^4\int_{0}^{T  } \widehat{a}^{-1}\left\{\phi_{\hbar, \lef}(\partial_{\sigma}\phi_{\hbar, \righ})-\phi_{\hbar, \righ}(\partial_{\sigma}\phi_{\hbar, \lef})\right\}(0, \tau)d\tau\,,\\
w^{\mathsf{d}}_{\lef,\righ}(\hbar)&=-\hbar^4\int_{0}^{T  } \widehat{a}^{-1}\left\{\phi_{\hbar, \lef}(\partial_{\sigma}\phi_{\hbar, \righ})-\phi_{\hbar, \righ}(\partial_{\sigma}\phi_{\hbar, \lef})\right\}(-L, \tau)d\tau\,.
\end{align*}
Using Propositions~\ref{prop.WKB} and \ref{prop:WKB=gs}, and the fact that $\phi_{\hbar,\lef}(\sigma,\tau)=\phi_{\hbar,\righ}(-\sigma,\tau)$, we write,
\[w^{\mathsf{u}}_{\lef,\righ}(\hbar)=
\Big(2\hbar^{5/2}|\xi_{0,\lef}(0)|^2\Phi_{\lef}'(0)+\mathcal O(\hbar^{7/2})\Big)\exp\left(\frac{-\mathsf{S}_{\mathsf{u}}}{\hbar}\right)\,.\]
In the same way, we find,
\[w^{\mathsf{d}}_{\lef,\righ}(\hbar)=
\Big(-\hbar^{5/2}\xi_{0,\lef}(-L)\xi_{0,\lef}(L)\big(\Phi'_{\lef}(L)+\Phi'_{\lef}(-L)\big)+\mathcal O(\hbar^{7/2})\Big)\exp\left(\frac{-\mathsf{S}_{\mathsf{d}}}{\hbar}\right)\,.\]
The  computation of $w^{\mathsf{d}}_{\lef,\righ}$ is easy by using the expressions of $\xi_{0,\lef}$ and $\Phi_{\lef}$ in Proposition~\ref{prop.WKB}. In this way, we get,
$$w^{\mathsf{d}}_{\lef,\righ}=\left[-2\hbar^{5/2}\left(\frac{\gamma}{\pi}\right)^{1/2}\sqrt{\mathfrak{v}(L)}\,\exp\left(-\int_{s_{\lef}}^L\frac{(\sqrt{\mathfrak v})'-\gamma}{\sqrt{\mathfrak v}}\,d\sigma\right)+\mathcal O(\hbar^{7/2})\right]\exp\left(\frac{-\mathsf{S}_{\mathsf d}}{\hbar}\right)\,.$$
Note that $\mathfrak v(-L)=\mathfrak v(L)$ by periodicity. In other words, we are saying that $s=\pm L$ defines the same point on the boundary, see Figure~1.

To compute $w^{\mathsf{d}}_{\lef,\righ}$ we use the two symmetry properties $\xi_{0,\lef}(\sigma)=\xi_{0,\righ}(-\sigma)$ and $\Phi_{\lef}(\sigma)=\Phi_{\righ}(-\sigma)$ and the expresions of $\xi_{0,\righ}$ and $\Phi_{\righ}$ in Proposition~\ref{prop.WKB}. We obtain,
$$w^{\mathsf{u}}_{\lef,\righ}=\left[-2\hbar^{5/2}\left(\frac{\gamma}{\pi}\right)^{1/2}\sqrt{\mathfrak{v}(0)}\,\exp\left(-\int_{s_{\righ}}^0\frac{(\sqrt{\mathfrak v})'+\gamma}{\sqrt{\mathfrak v}}\,d\sigma \right)+\mathcal O(\hbar^{7/2})\right]\exp\left(\frac{-\mathsf{S}_{\mathsf u}}{\hbar}\right)\,.$$
By adding the expressions of $w^{\mathsf{u}}$ and $w^{\mathsf{d}}$, we get an expression consistent with the one in \eqref{formspl1}. Recalling that $\hbar=h^{1/4}$, we finish the proof of \eqref{theo.main} by using Propositions \ref{prop:red-bnd} and \ref{prop.splittingLhat}.

\section{A Weyl formula}\label{subsec.BO}
This section is devoted to the proof of the following theorem (see \cite{Fr, Ka-rmp} for similar results for the Schr\"odinger operator with magnetic fields).
\subsection{Main result}
For $\lambda\in\R$, we denote by
\[\mathsf{N}\left(\mathcal{L}_{h},\lambda\right)={\rm Tr}\Big(\mathcal L_h\,\mathbf 1_{(-\infty,\lambda]}(\mathcal L_h)\Big)\,,\]
the number of eigenvalues $\mu_n(h)$ of $\mathcal L_h$ below the energy level $\lambda$.

\begin{theorem}\label{prop.weyl}
Under Assumption \ref{hyp.main1}, we have
\begin{enumerate}[\rm i.]
\item  the Weyl estimate of the semiclassically negative eigenvalues: 
\begin{equation}\label{eq.Weyl*}
\forall \Lambda\in(0,1)\,,\, \mathsf{N}\left(\mathcal{L}_{h},-\Lambda h\right)\underset{h\to 0}=
\frac{|\Gamma|\sqrt{1-\Lambda}}{\pi h^{\frac12}}+\mathcal O(1)\,;
\end{equation}
\item  the Weyl estimate of the low lying eigenvalues: 
\begin{equation}\label{eq.Weyl}
\forall E \in \mathbb R\,,\, \mathsf{N}\left(\mathcal{L}_{h},-h+Eh^{\frac{3}{2}}\right)\underset{h\to 0}{\sim}\frac{1}{\pi h^{\frac{1}{4}}}\int_{\Gamma}\sqrt{(E+\kappa)_{+}}\,\, ds(x)\,.
\end{equation}
\end{enumerate}
 \end{theorem}
 
The proof of Theorem~\ref{prop.weyl} relies on a comparison with an effective Hamiltonian (see Proposition~\ref{prop.normalform} below). 
\begin{remark}
The counting of eigenvalues for the Robin problem appears (at least in the case of the disk) in the thesis of A. Stern \cite{St} in 1925 but note that the author (who refers to the book by Pockels \cite{Po91} (written at the end of the nineteen-th century) is only counting the total number of negative eigenvalues. In this case, this is directly related to the counting function for the Dirichlet-to-Neumann operator. We refer to \cite{GiPo} for a recent survey on these questions.
\end{remark}

\subsection{More about the Robin 1D-Laplacian}
This subsection contains one key element in the proof of Theorem~\ref{prop.weyl} obtained through an additional analysis of the weighted operator in \eqref{eq:H0b} and its groundstate. Note that the analysis of this operator  is equivalent to that of the operator $\widetilde{\mathcal{H}}_{B}^{\{T\}}$ defined in \eqref{eq.tildeH}. Recall that the operator $\widetilde{\mathcal{H}}_{B}^{\{T\}}$ is defined in the interval  $(0,T)$ and that its ground state $\tilde u_{B}^{\{T\}}$ is given by the relation
$$\tilde u_{B}^{\{T\}}=(1-B\tau)^{\frac{1}{2}}u_{B}^{\{T\}}$$
where $ u_{B}^{\{T\}}$ is the groundstate of the operator in \eqref{eq:H0b}. Since $ \tilde u_{B}^{\{T\}}$ is normalized in $L^2(0,T)$,
\begin{equation}\label{remorth}
 \int _0^T \partial_{B}  \tilde u_{B}^{\{T\}}\, \tilde u_{B}^{\{T\}}\,d\tau =0\,.
 \end{equation}

\noindent For further use, we would like to estimate $\|\partial_{B} \tilde u_{B}^{\{T\}}\|_{L^2((0,T), d\tau)}$, uniformly with respect to $B$ and $T$.

\begin{lemma}\label{lem.h12}
There exist $C>0$  and $T_0 >0$ such that for all $T\geq T_0$ and $B\in(-1/(3T), 1/(3T))$,
\begin{align}
\left|\partial_{B}\lambda_{1}\left(\widetilde{\mathcal{H}}_{B}^{\{T\}}\right)\right|&\leq C\label{FHb}\,,\\
\|\partial_{B} \tilde u_{B}^{\{T\}}\|_{L^2((0,T),d\tau)}&\leq C\label{FHc}\,.
\end{align}
\end{lemma}
\begin{proof}
We recall that $\tilde q^{\{T\}}_{B}$ is defined in \eqref{eq.qtilde} (the associated bilinear form is denoted in the same way). From the eigenvalue equation we get, for all $\varphi\in H^1(0,T)$,
\begin{equation}\label{eq.eigen-var}
\tilde q^{\{T\}}_{B}(\tilde u_{B}^{\{T\}},\varphi)-\lambda_{1}\langle\tilde u_{B}^{\{T\}},\varphi\rangle=0\,.
\end{equation}
Then, we take the derivative with respect to $B$ and we get
\begin{equation}\label{eq.derivB}
\tilde q^{\{T\}}_{B}(\partial_{B}\tilde u_{B}^{\{T\}},\varphi)-\lambda_{1}\langle\partial_{B}\tilde u_{B}^{\{T\}},\varphi\rangle=\partial_{B}\lambda_{1}\langle\tilde u_{B}^{\{T\}},\varphi\rangle-\partial_{B}\tilde q^{\{T\}}_{B}(\tilde u_{B}^{\{T\}},\varphi)\,.
\end{equation}
We take $\varphi=\tilde u_{B}^{\{T\}}$ and the l.h.s. in \eqref{eq.derivB} vanishes (use $\varphi=\partial_{B}\tilde u_{B}^{\{T\}}$ in \eqref{eq.eigen-var}\footnote{Note that $\partial_{B}  \tilde u_{B}^{\{T\}}$ belongs to the form domain, but not to the domain of the operator. We have indeed
\begin{equation*}\label{eq.reste-IBP}
\left(\partial_B\tilde u_{B}^{\{T\}}\right)'(0) =\partial_B\left(\tilde u_{B}^{\{T\}}\right)'(0)=\left(-1-\frac{B}{2}\right)\partial_{B}\tilde u_{B}^{\{T\}}(0)-\frac{\tilde u_{B}^{\{T\}}(0)}{2}\,.
\end{equation*}}). Then, the r.h.s. vanishes and we deduce the Feynman-Hellmann formula
\begin{equation}\label{eq.FH+boundary}
\partial_{B}\lambda_{1}\left(\widetilde{\mathcal{H}}_{B}^{\{T\}}\right)=-\int_{0}^{T}\partial_{B}\left(\frac{B^2}{4(1-B\tau)^2}\right)  \tilde u_{B}^{\{T\}}\,\tilde u_{B}^{\{T\}}d\tau-\frac{1}{2}|\tilde u^{\{T\}}_{B}(0)|^2\,.
\end{equation}
A $T$-uniform continuous Sobolev embedding (for $T\geq 1$) and Proposition \ref{lem:H0b;l} give
\begin{equation}\label{eq.boundu(0)}
|\tilde u^{\{T\}}_{B}(0)|\leq C\, \|\tilde u^{\{T\}}_{B}\|_{H^1(0,T) }\leq\tilde C\,.
\end{equation}
Therefore \eqref{FHb} holds thanks to \eqref{eq.FH+boundary}, Proposition \ref{prop.AgmonuBT} and \eqref{eq.boundu(0)}.

Now we take $\varphi=\partial_{B}  \tilde u_{B}^{\{T\}}$ in \eqref{eq.derivB} and we find with the same considerations (and \eqref{remorth}):
\begin{align*}
\tilde q^{\{T\}}_{B}(\partial_B\tilde u_{B}^{\{T\}})-\lambda_{1}\left(\widetilde{\mathcal{H}}_{B}^{\{T\}}\right)\|\partial_{B}\tilde u_{B}^{\{T\}}\|^2_{L^2}&\leq C(\|\partial_{B}\tilde u_{B}^{\{T\}}\|_{L^2}+|\tilde u_{B}^{\{T\}}(0)||\partial_{B}\tilde u_{B}^{\{T\}}(0)|)\\
&\leq  C\|\partial_{B}\tilde u_{B}^{\{T\}}\|_{L^2}+\tilde C|\partial_{B}\tilde u_{B}^{\{T\}}(0)|\,.
\end{align*}
With the spectral gap (see Lemmas \ref{lem:1DL} and
\ref{lem:H0b}) together with   \eqref{remorth}, we get,
\begin{equation}\label{eq.gap'}
\|\partial_{B}\tilde u_{B}^{\{T\}}\|^2_{L^2}\leq C+C\, |\partial_{B}\tilde u_{B}^{\{T\}}(0)|\,,
\end{equation}
and thus
\[\tilde q^{\{T\}}_{B}(\partial_B\tilde u_{B}^{\{T\}})\leq  C+C\, |\partial_{B}\tilde u_{B}^{\{T\}}(0)|\,.\]
From this  we deduce
\[\|\partial_{B}\tilde u_{B}^{\{T\}}\|_{H^1(0,T)}\leq C+C\, \|\partial_{B}\tilde u_{B}^{\{T\}}\|_{L^2(0,T)}+C\, |\partial_{B}\tilde u_{B}^{\{T\}}(0)|^{\frac{1}{2}} \]
and,  by Sobolev embeddings, 
\[|\partial_{B}\tilde u_{B}^{\{T\}}(0)|\leq \widehat C+ \widehat C\, \|\partial_{B}\tilde u_{B}^{\{T\}}\|_{L^2(0,T)}\,.\]
The estimate \eqref{FHc} follows from \eqref{eq.gap'}.
\end{proof}

\subsection{Proof of the Weyl formulas}
Thanks to the min-max principle  and the usual Weyl formula in dimension one for the operator on the circle $h^\frac 12 D_\sigma^2- \kappa(\sigma)$   (use a direct comparison with the case with constant potential for \rm{(i)} and use for example  \cite{Ro87} for the case \rm{(ii)}\,), Theorem~\ref{prop.weyl} is a consequence of the following proposition which permits to localize the eigenvalues $\mu_n$ of $\mathcal L_h$ by comparison with effective Hamiltonians.

\begin{proposition}\label{prop.normalform}
Under Assumption \ref{hyp.main1},  for $\epsilon_{0}\in(0,1)$, $h>0$, we let\[\mathcal{N}_{\epsilon_{0}, h}=\{n\in\mathbb{N}^* : \mu_{n}(h)\leq-\epsilon_{0}h\}\,.\]
There exist positive constants $h_{0}, C_{+}, C_{-}$ such that, for all $h\in(0,h_{0})$ and $n\in\mathcal{N}_{\epsilon_{0}, h}$,
 \begin{equation}\label{eq:App.eh}
\mu^-_{n}(h)\leq\mu_{n}(h)\leq \mu^{+}_{n}(h)\,,\end{equation}
where $\mu^{\pm}_{n}(h)$ is the $n$-th eigenvalue of $\mathcal{L}^{\mathsf{eff}, \pm}_{h}$ defined by
 \[\mathcal{L}^{\mathsf{eff},+}_{h}=-h+(1+C_{+}h^{\frac{1}{2}})h^2D_{\sigma}^2-\kappa(\sigma)h^{\frac{3}{2}}+C_{+}h^2\,,\]
and
 \[\mathcal{L}^{\mathsf{eff},-}_{h}=-h+(1-C_{-}h^{\frac{1}{2}})h^2 D_{\sigma}^2-\kappa(\sigma)h^{\frac{3}{2}}-C_{-}h^{2}\,.\]
 \end{proposition}
\subsection{Proof of Proposition  \ref{prop.normalform}}
The proof will be done in three steps.
\subsubsection{Preliminary considerations}
Thanks to the Agmon estimates established in Section \ref{sec.tubular}, it is sufficient to work with $\widehat{\mathcal{L}}_{\hbar}$. As suggested by the proof of Lemma \ref{lem:lb-qf}, the spectral analysis of $\widehat{\mathcal L}_\hbar$ may be done with the Born-Oppenheimer strategy. Let us recall the expression of the quadratic form $\widehat{\mathcal Q}_{\hbar}$, defined in  \eqref{eq:dom-Lh-hat},
\[\widehat{\mathcal Q}_{\hbar}(\psi)=\int_{-L}^L \int_{0}^{T}\widehat a^{-2}\hbar^4|\partial_\sigma \psi|^2\,d\tau d\sigma+\int_{-L}^L \left\{\int_{0}^{T}|\partial_\tau \psi|^2\widehat a \, d\tau- |\psi(\sigma,0)|^2\right\}\,d\sigma\,,\]
 with $T=D\hbar^{-1}$.
We let also 
 $$
 \mathcal{H}_{\kappa(\sigma),\hbar} = \mathcal H^{\{T\}}_{B}\,,
 $$
 with $B=h^{\frac{1}{2}}\kappa(\sigma)=\hbar^2\kappa(\sigma)$.
 
We introduce for $\sigma \in [-L,L)$ the Feshbach projection $\Pi_\sigma$  on the normalized groundstate of  $ \mathcal{H}_{\kappa(\sigma),\hbar}$, denoted by $v_{\kappa(\sigma), \hbar}$,
\[
\Pi_\sigma  \psi=\langle\psi, v_{\kappa(\sigma), \hbar}\rangle_{L^2((0,T),\widehat a d\tau)} v_{\kappa(\sigma), \hbar}\,.
\]

We also let
\[\Pi_\sigma^\perp=\mathsf{Id}-\Pi_\sigma \]
and
\begin{equation}\label{defR}
\qquad R_{\hbar}(\sigma)=\|\partial_{\sigma} v_{\kappa(\sigma), \hbar}\|^2_{L^2((0,T),\,\widehat a d\tau)}\,.
\end{equation}
The quantity $R_{\hbar}$ is sometimes called \enquote{Born-Oppenheimer correction}.

\noindent To be reduced to classical considerations, the main point is to control the effect of replacing $\widehat{a}^{-2}$ by $1$. 
\begin{lemma}\label{lem.ahat1}
We have, for all $\psi\in\Dom(\widehat{\mathcal Q}_{\hbar})$,
\begin{align*}
&\left|\int_{\widehat{\mathcal V}_{T}} \widehat{a}^{-2}|\partial_{\sigma}\psi|^2\widehat{a} \, d\sigma d\tau-\int_{\widehat{\mathcal V}_{T}} |\partial_{\sigma}\psi|^2\widehat{a}\, d\sigma d\tau\right|\\
&\leq  \widetilde C \int_{\widehat{\mathcal V}_{T}} \hbar^{2} |f_\psi '(\sigma)|^2+\hbar R_{\hbar}(\sigma)|f_\psi(\sigma)|^2+ \hbar |\partial_{\sigma}\Pi_\sigma^\perp  \psi|^2 d\sigma d\tau\,.
\end{align*}
with
$$
f_\psi (\sigma):=\langle\psi (\sigma,\cdot) , v_{\kappa(\sigma), \hbar}\rangle_{L^2((0,T),\,\widehat a d\tau)}\,.
$$

\end{lemma}
\begin{proof}
We write
\begin{align*}
&\left|\int_{\widehat{\mathcal V}_{T}} \widehat{a}^{-2}|\partial_{\sigma}\psi|^2 \, d\sigma d\tau-\int_{\widehat{\mathcal V}_{T}} |\partial_{\sigma}\psi|^2\, d\sigma d\tau\right|\\
&\leq C\int_{\widehat{\mathcal V}_{T}} \hbar^{2}\tau|\partial_{\sigma}\psi|^2 d\sigma d\tau\\
&\leq 2 C \int_{\widehat{\mathcal V}_{T}} \hbar^{2}\tau\left(|\partial_{\sigma}\Pi_\sigma \psi|^2+|\partial_{\sigma}\Pi_\sigma^\perp  \psi|^2\right) d\sigma d\tau\\
&\leq  \widetilde C \int_{\widehat{\mathcal V}_{T}} \hbar^{2} |f'_\psi(\sigma)|^2+\hbar R_{\hbar}(\sigma)|f_\psi(\sigma)|^2+\hbar|\partial_{\sigma}\Pi_\sigma^\perp  \psi|^2 d\sigma d\tau\,,
\end{align*}
  where we used that
\begin{equation}\label{Agma}
\int_{0}^{T}\tau |v_{\kappa(\sigma), \hbar}|^2 d\tau\leq C
\end{equation}
 (that is a consequence of Proposition \ref{prop.AgmonuBT}) and that $\tau \hbar^2$ may be estimated by $T \hbar^2=D \hbar$.
\end{proof}

\begin{lemma}\label{lem.normali}
We have
\[\int_{0}^{T} v_{\kappa(\sigma), \hbar}\partial_{\sigma} v_{\kappa(\sigma), \hbar}\, \widehat{a}\, d\tau=\mathcal{O}(\hbar^2)\,.\]
\end{lemma}
\begin{proof}
We notice from the normalization of  $ v_{\kappa(\sigma), \hbar}$ that
\[\partial_{\sigma}\int_{0}^T v_{\kappa(\sigma),\hbar} v_{\kappa(\sigma), \hbar}\widehat{a}\,d \tau=0\,,\]
so that
\[2\int_{0}^{T} v_{\kappa(\sigma), \hbar}\partial_{\sigma} v_{\kappa(\sigma), \hbar}\, \widehat{a}\, d\tau=\int_{0}^T v_{\kappa(\sigma),\hbar} v_{\kappa(\sigma), \hbar}(\partial_{\sigma}\widehat{a})\,d \tau\,,\quad\mbox{ with }\partial_{\sigma}\widehat{a}=-\tau \hbar^2\kappa'(\sigma)\,.\]
The conclusion follows from \eqref{Agma}.
\end{proof}

\subsubsection{Upper and lower bounds}
Keeping these preliminaries in mind, the results below are consequences of almost the same computations as in \cite[Chapter 13]{Ray} (see also \cite{MT05}, \cite{DR14} where a similar strategy is used). The first follows from a computation using Lemmas \ref{lem.ahat1} and \ref{lem.normali}.
\begin{lemma}\label{prop.ub}
There exist $C>0$, $\hbar_{0}>0$ such that, for all $\psi\in\widehat{\mathcal{D}}_{T}$ and $\hbar\in(0,\hbar_{0})$, we have
\[\widehat{\mathcal Q}_{\hbar}(\Pi_\sigma \psi)\leq \int_{-L}^L \hbar^4(1+C\hbar^2)|f_\psi'(\sigma)|^2+\left(\hbar^4(1+C\hbar) R_{\hbar}(\sigma)+\lambda_{1}(\mathcal{H}_{\kappa(\sigma),\hbar})+C\hbar^6\right)|f_\psi(\sigma)|^2 d\sigma\,.\]
\end{lemma}
The next lemma is slightly more delicate. 
\begin{lemma}\label{prop.lb}
There exist $C>0$, $\hbar_{0}>0$ such that, for all $\psi\in\widehat{\mathcal{D}}_{T}$, $\epsilon\in\left(0,\frac{1}{2}\right)$ and $\hbar\in(0,\hbar_{0})$, we have
\begin{multline*}
\widehat{\mathcal{Q}}_{\hbar}(\psi)\geq\!\! \int_{-L}^L\!\! (1-\epsilon)(1-C\hbar^2)\hbar^4|f_\psi '(\sigma)|^2+\left\{\lambda_{1}(\mathcal{H}_{\kappa(\sigma),\hbar})-C(\epsilon^{-1} \hbar^4 R_{\hbar}(\sigma)+\epsilon^{-1}\hbar^8+\hbar^6)\right\}|f_\psi(\sigma)|^2 d \sigma\\
+\int_{-L}^L\!\! (1-\epsilon)(1-C\hbar)\hbar^4\Vert\partial_{\sigma}\Pi_\sigma ^\perp\psi\Vert_{L^2(\widehat a d\tau)}^2+\left\{\lambda_{2}(\mathcal{H}_{\kappa(\sigma),\hbar})-C\epsilon^{-1} \hbar^4 R_{\hbar}(\sigma)-C\epsilon^{-1}\hbar^8\right\}\Vert\Pi_\sigma ^\perp\psi\Vert_{L^2(\widehat a d\tau)}^2d\sigma\,.
\end{multline*}
\end{lemma}
\begin{proof}
First, we use Lemma \ref{lem.ahat1} to get that
\begin{equation}\label{eq.Qapp}
\widehat{\mathcal{Q}}_{\hbar}(\psi)\geq \widehat{\mathcal{Q}}^\mathsf{app}_{\hbar}(\psi)- \widetilde C\hbar^4 \int_{\widehat{\mathcal V}_{T}} \hbar^{2} |f_\psi '(\sigma)|^2+\hbar R_{\hbar}(\sigma)|f_\psi(\sigma)|^2+\hbar |\partial_{\sigma}\Pi_\sigma^\perp  \psi|^2 d\sigma d\tau\,,
\end{equation}
with
\[ \widehat{\mathcal{Q}}^\mathsf{app}_{\hbar}(\psi)=\int_{-L}^L \int_{0}^{T}\hbar^4|\partial_\sigma \psi|^2\widehat a \,d\tau d\sigma+\int_{-L}^L \left\{\int_{0}^{T}|\partial_\tau \psi|^2\widehat a \, d\tau- |\psi(\sigma,0)|^2\right\}\,d\sigma\,.\]
Then, we have the orthogonal decomposition
\begin{equation}\label{eq.ortho}
\int_{0}^T |\partial_{\sigma}\psi|^2\,\widehat{a}\,d\tau=\int_{0}^T |\Pi_\sigma \partial_{\sigma}\psi|^2\,\widehat{a}\,d\tau+\int_{0}^T |\Pi_\sigma ^\perp\partial_{\sigma}\psi|^2\,\widehat{a}\,d\tau\,.
\end{equation}
We have also the commutator identity
\begin{multline*}
[\partial_{\sigma}, \Pi_\sigma ]\psi=\langle\psi,\partial_{\sigma} v_{\kappa(\sigma), \hbar} \rangle_{L^2((0,T),\widehat a d\tau)} v_{\kappa(\sigma), \hbar} +\langle\psi,  v_{\kappa(\sigma), \hbar} \rangle_{L^2((0,T),\widehat a d\tau)}  \partial_{\sigma}v_{\kappa(\sigma), \hbar}\\
-\kappa'(\sigma)\hbar^{2}\left(\int_{0}^T \psi (\cdot,\tau) \, \tau v_{\kappa(\sigma), h}(\tau) \, d\tau\right) v_{\kappa(\sigma), \hbar}\,,
\end{multline*}
so that we get, by the Cauchy-Schwarz inequality, the estimate
\begin{equation}\label{eq.commutator}
\|[\partial_{\sigma}, \Pi_\sigma]\psi\|_{L^2((0,T),\widehat a d\tau)}\leq 2R_{\hbar}(\sigma)^{\frac{1}{2}}\|\psi\|_{L^2((0,T),\widehat a d\tau)}+C\hbar^2\|\psi\|_{L^2((0,T),\widehat a d\tau)}\,.
\end{equation}
For all $\epsilon\in(0,1)$, we get, by using the classical inequality $|a-b|^2\geq (1-\epsilon)a^2-\epsilon^{-1}b^2$ and \eqref{eq.ortho}, 
\[\int_{0}^T |\partial_{\sigma}\psi|^2\,\widehat{a}\,d\tau\geq (1-\epsilon)\left\{ \int_{0}^T |\partial_{\sigma}\Pi_\sigma \psi|^2\,\widehat{a}\,d\tau+\int_{0}^T |\partial_{\sigma}\Pi_\sigma ^\perp\psi|^2\,\widehat{a}\,d\tau\right\} -2\epsilon^{-1}\int_{0}^T |[\partial_{\sigma},\Pi_\sigma ]\psi|^2\,\widehat{a}\,d\tau\,.\]
With \eqref{eq.commutator}, we get
\begin{multline}\label{eq.1}
\int_{0}^T \hbar^4|\partial_{\sigma}\psi|^2\,\widehat{a}\,d\tau\geq (1-\epsilon)\hbar^4\left\{ \int_{0}^T |\partial_{\sigma}\Pi_\sigma \psi|^2\,\widehat{a}\,d\tau+\int_{0}^T |\partial_{\sigma}\Pi_\sigma ^\perp\psi|^2\,\widehat{a}\,d\tau\right\} \\
-C\epsilon^{-1}(\hbar^4 R_{\hbar}(\sigma)+\hbar^8)\|\psi\|^2_{L^2((0,T),\widehat a d\tau)}\,.
\end{multline}
By computing and using Lemma \ref{lem.normali} to deal with the double product, we have
\begin{equation}\label{eq.2}
\int_{0}^T |\partial_{\sigma}\Pi_\sigma \psi|^2\,\widehat{a}\,d\tau\geq(1-C\hbar^2)|f_\psi '(\sigma)|^2+(R_{\hbar}(\sigma)-C\hbar^2)|f_\psi(\sigma)|^2\,.
\end{equation}
Moreover we have, by an orthogonal decomposition and the min-max principle,
\begin{equation}\label{eq.3}
\int_{0}^{T}|\partial_\tau \psi|^2\widehat a \, d\tau- |\psi(\sigma,0)|^2\geq \lambda_{1} (\mathcal{H}_{\kappa(\sigma),\hbar}) |f_\psi(\sigma)|^2+\lambda_{2} (\mathcal{H}_{\kappa(\sigma),\hbar})\|\Pi_\sigma ^\perp\psi\|^2_{L^2((0,T),\widehat a d\tau)}
\end{equation}
The conclusion follows from \eqref{eq.Qapp}, \eqref{eq.1}, \eqref{eq.2}, \eqref{eq.3} and by integrating with respect to $\sigma$.
\end{proof}
\subsubsection{End of the proof of Proposition \ref{prop.normalform}}

We apply Lemmas \ref{prop.ub} and \ref{prop.lb} with $\epsilon=\hbar^{2}$. Then, we use Lemmas \ref{lem:1DL} and \ref{lem:H0b} and Proposition \ref{lem:H0b;l} to deduce that
\[\lambda_{1}(\mathcal{H}_{\kappa(\sigma),\hbar})=-1-\hbar^2\kappa(\sigma)+\mathcal{O}(\hbar^4)\,,\]
and that there exist $h_{0}>0$ and $C>0$ such that, for all $h\in(0,h_{0})$,
\[\lambda_{2}(\mathcal{H}_{\kappa(\sigma),\hbar})\geq -C\hbar>-\frac{\epsilon_{0}}{2}\,.\]
Then we notice that $R_{\hbar}(\sigma)$ (introduced in \eqref{defR}) satisfies $R_\hbar(\sigma)=\mathcal{O}(\hbar^4)$ thanks to Lemma~\ref{lem.h12} and the relation $B=\kappa(\sigma) \hbar^2$. The conclusion comes from the min-max principle (see \cite[Chapter 13]{Ray}): the lower bounds in Theorem \ref{prop.weyl} follow from Lemma \ref{prop.ub} and the upper bounds from Lemma \ref{prop.lb}.

\begin{remark}
One can see that Proposition \ref{prop.normalform} only requires that the boundary is $\mathcal{C}^2$ and that its curvature is Lipschitzian (that is an admissible boundary of order at least $3$ in the sense of \cite{PP-eh}). This result matches with the one of \cite{PP-eh}. Moreover, our effective Hamiltonians provide a uniform approximation valid for all the eigenvalues less than the energy level $-\epsilon_{0}h$ and not only for an $h$-independent number of low-lying eigenvalues. The underlying operator reduction follows from the general arguments often used in the Born-Oppenheimer framework. One can reasonably hope to extend the analysis to higher dimensional situations and improve the spectral approximations of \cite{PP-eh} obtained in the case of admissible boundaries of order $3$.
\end{remark}

\subsection*{Acknowledgments}
A.K. is supported by a grant from Lebanese University. This work was
partially supported by the ANR (Agence Nationale de la Recherche),
project {\sc Nosevol}\break  n$^{\rm o}$ ANR-11-BS01-0019 and by the
Henri Lebesgue Center (programme  \enquote{Investissements d'avenir} \break -- n$^{\rm
  o}$ ANR-11-LABX-0020-01). B.H. is grateful to the Erwin Schr\"odinger Institute in Vienna where this paper was achieved.

\end{document}